\documentclass[11pt, reqno]{amsart}
\usepackage{latexsym}
\usepackage{amssymb}
\usepackage{amsmath}
\usepackage{amscd}
\usepackage{amsthm}
\usepackage{graphicx}
\usepackage{color}
\usepackage{bm}

\newcommand\dela[1]{}
\allowdisplaybreaks

\def\XXint#1#2#3{{\setbox0=\hbox{$#1{#2#3}{\int}$ }
\vcenter{\hbox{$#2#3$ }}\kern-.6\wd0}}

\newtheorem{theorem}{Theorem}
\numberwithin{equation}{section}
\numberwithin{theorem}{section}
\newtheorem{lemma}[theorem]{Lemma}

\newtheorem{remark}[theorem]{Remark}

\numberwithin{equation}{section}

\newcommand{\wt}{\widetilde}
\newcommand{\e}{\varepsilon}

\def\P{\mathbb{ P}}
\def\E{\mathbb{E}}
\def\P1{\mathbb{P}^{1}}

\def\o{\overline}
\begin{document}


\title{ Stochastic homogenization for a diffusion-reaction model}

\author[H. Bessaih]{Hakima Bessaih}
\address{University of Wyoming, Department of Mathematics \& Statistics,
East University Avenue, Laramie WY 82071, United States}
\email{ bessaih@uwyo.edu}

\author[Y. Efendiev]{ Yalchin Efendiev}
\address{Department of Mathematics \& ISC, 
Texas A\&M University}
\email{efendiev@math.tamu.edu}

\author[R. F. Maris]{Razvan Florian Maris}
\address{Alexandru Ioan Cuza University of Iasi, Faculty of Economics and Business Administration, Bd. Carol I, 22, 
Iasi 700505, 
Romania}
\email{florian.maris@feaa.uaic.ro}
\maketitle

{\footnotesize
\begin{center}


\end{center}
}
\begin{abstract}
In this paper, we study stochastic homogenization of a coupled
diffusion-reaction
system. The diffusion-reaction system is coupled to stochastic differential
equations, which govern the changes in the media properties.
Though homogenization with changing media properties has been studied
in previous findings, there is 
little research on homogenization when the media properties
change due to stochastic differential equations. Such processes 
occur in many applications, where the changes in media properties
are due to particle deposition. 
In the paper, we investigate the well-posedness of the nonlinear
fine-grid (resolved) problem and derive limiting equations. 
We formulate the cell problems and derive the limiting equations,
which are deterministic with nonlinear reaction terms. 
The limiting equations involve the invariant measures corresponding
to stochastic differential equations.
These obtained results can play an important role 
for modeling in porous media and
allow the use of simplified and deterministic limiting equations.

\end{abstract}

\maketitle

\section{Introduction and formulation of the problem}
\label{sec1}




Fluid flow through a porous media is a subject of wide interest 
that has been widely studied in the past years. It has many applications 
in  real life problems like energy, biology and material sciences 
to quote just a few. These models typically contain many different 
spatial and temporal scales. 
Various phenomena are modeled by partial differential equations 
that include coefficients describing the porosity, permeability 
and diffusion processes.  Though many static problems are well studied
for these applications, the problems with dynamically changing media
properties are much less studied research area. Many research in
this direction includes smoothly and deterministically changing
permeability fields; however, in many real-world
applications, the permeability changes occur due to particle deposition.
This is a challenging problem as the stochastic differential equations
are tightly coupled to porous media equations
and govern the permeability
changes.

In this paper, we consider the following system:
\label{subs12}
\begin{equation}
\label{system1}
\left\{
\begin{array}{rll}
\dfrac{\partial u^\e}{\partial t} (t, x)&= \operatorname{div} \left(A\left(\dfrac{x}{\e}\right) \nabla u^\e(t, x) \right) + \alpha \left(\dfrac{x}{\e}, v^\e(t, x)\right) u^\e(t, x) + f(t, x) &\mbox{ in }\ [0, T] \times D, \\
d v^\e(t, x) &= -\dfrac{1}{\e} (v^\e(t, x)-u^\e(t, x))dt + \sqrt {\dfrac{Q}{\e}} dW(t, x) &\mbox{ in }\ [0, T] \times D, \\
u^\e(t, x) &=0 &\mbox{ on }\ [0, T]\times\partial D, \\
u^\e(0, x) & = u_0^\e(x)&\mbox{ in }\ D, \\
v^\e(0, x) & = v_0^\e(x)&\mbox{ in }\ D, 
\end{array}
\right. 
\end{equation}
where $D$ is a bounded domain of $\mathbb{R}^3$ with a smooth boundary $\partial{D}$,  $u^\e$ is the fluid velocity and and $v^\e$ is the particle velocity. Moreover,  $W(t)$ is an $L^2(D)$-valued standard Brownian motion defined on a complete probability basis $(\Omega, \mathcal{F}, (\mathcal{F}_t)_t, {\mathbb P})$ with expectation $\E$, and $Q$ is a bounded linear operator on $L^2(D)$ of trace class. $u^\e_0$ and $v^\e_0$ are the initial conditions and $f$ is an external force. 

Our model describes an equation with a diffusion $A(y)$ that has 
heterogeneous properties and such that the heterogeneous reaction
$\alpha(y,\cdot)$ is affected by particle deposition in the medium. 
These particles have a faster motion than the motion of the fluid flow and 
are driven by a stochastic perturbation of Brownian type.   
A simpler version of this model has been studied in \cite{BEM1},
 where the diffusion $A$ was considered to be constant and equal to 1. 
The heterogeneous diffusion brings an additional difficulty and makes
this problem more realistic since one deals with heterogeneous permeability
fields in most porous media problems.
 
Our main goal in this paper is to study the asymptotic behavior of the solutions of system \eqref{system1} when $\e\rightarrow 0$. Notice that $u^\e$ is random through the function $\alpha$ that depends on the stochastic process $v^\e$ solution of a stochastic differential equation. Moreover, the function $\alpha$ and the matrix $A=\left(a_{ij}\right)_{1\leq i, j \leq 3}$ are multiscale. Here, $u^\e$ is the slow component and $v^\e$ is the fast one. We will prove that $u^\e$ converges to an averaged velocity $\o{u}$ solution of the averaged equation \eqref{eqou1} where the averaged operators $\o{A}$ and $\o\alpha$ are given by \eqref{homop} and \eqref{defoalpha}. Here, the averages are taken with respect to the periodic variable $y$ and the invariant measure associated to the process $v^\e$ for a frozen $u^\e$ and the averaged operator $\o{A}$ is defined in terms of $\chi$ the solution of the cell problem given in Section 4.

For $\e>0$ fixed, the well posedness of system \eqref{system1} does not follow from classical results and has to be studied accordingly. In this paper, we assume that $\alpha(y,\cdot)$ is bounded and Lipschitz uniformly with respect to the variable $y$. In particular, the uniqueness of solutions is proved by using successive estimates in order to get to apply Gronwall Lemma, see Section 3.1. for more details. 

We prove the existence of weak solutions by using a Galerkin approximation $(u_{n}^\e, v_{n}^\e)$
that is a solution of a well posed system and then pass to the limit on $n$ after performing some uniform estimates in $n$. These estimates are also uniform in $\e$. 
By using our assumption on $\alpha$ and the special form of our system, we are able to prove the uniqueness of the weak solution $(u^\e, v^\e)$. We prove that our weak solution is also strong, and get better uniform estimates in $\e$ for the solution $u^\e$ in the Sobolev space $W^{1, 2}(0, T; L^2(D))$. 

We define the associated cell problem.  Then, we study the asymptotic behavior of the fast motion variable $v^\e$ for a frozen slow motion variable $u^\e$. Indeed, we consider the SDE \eqref{v} for a given $\xi$. It has a mild solution which is also a strong solution. Its transition semigroup $P_t^\xi$ is well defined and has a unique invariant measure 
$\mu^{\xi}$ which is ergodic and strongly mixing. 
The operator $ \alpha^\e$ is defined in Section 3 while the operator 
$\o{\alpha^\e}$ is defined in Section 5 and refers to the average of $\alpha^\e$ wrt to the invariant measure 
$\mu^{\xi}$. The main difficulty in showing the convergence stands in passing to the limit on the term
\begin{equation}\label{term}
\begin{split}
\int_D \left( \alpha^\e(v^\e(t)) u^\e(t) - \o{\alpha}(\o{u}(t)) \o{u}(t) \right)\phi dx
\end{split}
\end{equation}
for $\phi \in H_0^1(D)$, where $\o\alpha$ is defined in \eqref{defoalpha}. 

There is a quite large number of papers dealing with averaging principles for finite dimensional systems in both deterministic and stochastic systems.
Less has been done in the infinite dimensional setting, we refer to \cite{ Cerrai2009, CF2008} and the references therein.  There is not much in the literature dealing with averaging systems for porous media when spatial heterogeneities are present. We refer to our previous paper \cite{BEM1}, where to our knowledge, it was the first paper where time and spatial scales have been considered for porous media in a stochastic setting.  In \cite{CF2008}, the authors prove an averaging principle for a very general class of stochastic PDEs. 
Our system looks similar to theirs with a very important difference. Our function $(u,v)\to\alpha(\cdot, v) u$ is not Lipchitz and it contains the variable $x/\e$ that describes the heterogeneities of the medium. 
Hence, their results, although very general,
 could not be used in  \cite{BEM1} nor in the current paper.

Our model \eqref{subs12} is a generalization of the model considered in \cite{BEM1} since it contains a diffusion coefficient  $A(\frac{x}{\e})$ that is 
heterogeneous in space. It was equal to 1 in the previous paper  \cite{BEM1}. The presence of this coefficient does not affect the regularity of the solution $(u^\e, v^\e)$ but it does affect the uniform estimates that we can obtain. This is the main reason,  we are not able to use the method previously used in \cite{BEM1} that consisted in applying the It\^o formula on the $\Psi^\e (u^\e,v^\e)$, where 

\begin{equation*}
\Psi^\e (\eta,\xi) = \int_0^\infty e^{-c(\e) t} P_t^\xi \left[ \int_D \left(\alpha^\e(\cdot) - \o{\alpha^\e}(\xi)\right)\xi\phi dx (\eta)\right] dt.
\end{equation*}
Isolating the term \eqref{term}, a uniform estimate in $H^{2}(D)^3$ was needed to be able to pass to the limit. 
Unfortunately, while  the random variable $u^\e$ in \cite{BEM1} was uniformly bounded in the Sobolev space $H^{2}(D)^3$, in the current paper it is only uniformly bounded in the Sobolev space $H^{1}(D)^3$.

Instead, the limit in the term $\eqref{term}$ will be performed by using a Khasminskii type argument, following an idea already introduced in \cite{Cerrai2009} where as mentioned earlier,  our term $\alpha (\cdot, v^\e)u^\e$ does not satisfy the same assumptions. Hence, their method can't be adapted as is for the model \eqref{subs12}. In particular,  we need to apply the semigroup $P_t^\xi$ to a function
of the form 

\begin{equation*}\label{deffe}
F^\e(s, \eta)= \displaystyle\int_D \alpha^\e(\eta) u^\e \left(\e s\right)dx
\end{equation*}
and use the asymptotic properties of the semigroup, that is summarized in Lemma 5.4. The results of this lemma are not surprising but we were not able to find it in the literature. We believe that this is a nice new result that can be applied for other models.  

By using the uniform estimates obtained in Section 3 on the variable $u^\e$, a tightness argument 
and some known results for periodic functions, see \cite{A-2s} (lemma 1. 3) the passage to the limit is performed in distribution. We obtain a convergence in probability by using the fact that the limit $\o{u}$ is deterministic. 

The paper is organized as follows, Section 2 is dedicated to the introduction of the functional setting and assumptions. In Section 3, system \eqref{system1} is analyzed for every $\e>0$. In particular existence of strong solutions are established with their uniqueness and their uniform estimates with respect to $\e$.  We introduce the cell problem in Section 4. The fast motion variable $v^\e$ is analyzed in Section 5 where some known results are summarized with some references. In this section, the important Lemma 5.4 is given and proven in details since this is a crucial tool used to pass to the limit in the system. The passage to the limit is performed in Section 6. Furthermore, the well posedness of the averaged equation is established. 
 
\section{Preliminaries and Assumptions}
\label{sec2}
We make the following notations for spaces that will be used throughout the paper. 
For any two Hilbert spaces $X$ and $Y$, with norms denoted by $\|\cdot\|_X$ and $\|\cdot\|_Y$, $C(X, Y)$ denotes the space of continuous functions, and $C_b(X, Y)$ the Banach space of bounded and continuous functions
$\phi : X\to Y$ endowed with the supremum norm:
$$\|\phi\|_{C_b(X, Y)} = \sup_{x\in X} \|\phi (x)\|_{Y}. $$

For any $\phi\in C^u(X, Y)$, the subspace of uniformly continuous functions defined on $X$ with values in $Y$, we denote by $[\phi]_{C^u(X, Y)} : (0, \infty) \to \mathbb{R}$, the modulus of uniform continuity of $\phi$:
$$[\phi]_{C^u(X, Y)} (r) = \sup_{0 < \|x- y\|_X \leq r} \|\phi (x) - \phi(y)\|_Y, $$
with
$$\lim_{r\to 0} [\phi]_{C_b^u(X, Y)} (r) =0. $$
$Lip(X, Y)$ denotes the space of Lipschitz functions defined on $X$ with values in $Y$, for $\phi \in Lip(X, Y)$ we denote by $[\phi]_{Lip(X, Y)}$ the Lipschitz constant of $\phi$:
$$[\phi]_{Lip(X, Y)} = \sup_{x\neq y} \frac {\|\phi (x) - \phi(y)\|_Y}{\| x-y\|_X}. $$
We notice that for any $\phi \in Lip(X, Y)$ we have:
\begin{equation}\label{lip}
\|\phi (x)\|_Y \leq \|\phi (x) - \phi(0)\|_Y + \|\phi(0)\|_ Y\leq [\phi]_{Lip(X, Y)} \|x\|_X + \|\phi(0)\|_Y \leq ([\phi]_{Lip(X, Y)}+\|\phi(0)\|_Y) (1+\|x\|_X), 
\end{equation}
so the space will be naturally equipped with the norm
\begin{equation}\label{lipnorm}
\|\phi\|_{Lip(X, Y)} = \|\phi(0)\|_Y + [\phi]_{Lip(X, Y)}. 
\end{equation}
To simplify the notations, when there is no confusion we omit the use of subscripts from the notations, and we simply write $\|x\|$, $\|\phi\|$, $[\phi](r)$, $[\phi]$. Also if $Y=\mathbb{R}$ we omit it from the notations, and the spaces are denoted by $C(X)$, $C_b(X)$, $C^u(X)$, and $Lip(X)$. 

For $Y=[0, 1]^3$ the space $C_\#(Y)$ denotes the space of continuous functions on $Y$ that are $Y$-periodic and the space $L^2_\#(Y)$ denotes the closure of $C_\#(Y)$ in $L^2(Y)$. 

We will denote by $\langle\cdot, \cdot\rangle$ the inner product in $L^2(D)$. If we identify $L^2(D)$ with its dual $(L^2(D))'$ then we have the Gelfand triple $H_0^1(D)\subset L^2(D) \subset H^{-1}(D)$ with continuous injections. The dual pairing between $H_0^1(D)$ and $H^{-1}(D)$ will be also denoted by $\langle\cdot, \cdot\rangle$. 

We now give the assumptions for the system \eqref{system1} . 

The function $\alpha:Y\times\mathbb{R}\to\mathbb{R}$ satisfies the following conditions:

i) For any $\eta\in\mathbb{R}$ the function $\alpha(\cdot, \eta)$ is measurable. 

ii) For almost every $y\in Y$, the function $\alpha(y, \cdot)$ is bounded and Lipschitz, uniformly with respect to $y$. 

We notice that the function $\o{\alpha}:\mathbb{R}\to\mathbb{R}$, $\displaystyle\o{\alpha}(\eta) = \int_Y \alpha (y, \eta) dy$ is Lipschitz and bounded. 

The matrix $A=\left(a_{ij}\right)_{1\leq i, j \leq 3} \in L^\infty(Y; \mathbb{R}^{3\times 3})$ is strictly positive and bounded uniformly in $y\in Y$, i. e. there exist $0< m < M$ such that
\begin{equation}\label{A}
m \xi^2 \leq A(y)\xi \xi \leq M\xi^2, 
\end{equation}
for almost every $y\in Y$ and $\xi \in \mathbb{R}^3$. 

Throughout the paper, we assume that $f\in L^2(0, T;L^2(D))$ and $u^\e_0, v^\e_0 \in L^2(D)$ and that 
$W(t)$ is an $L^2(D)$-valued standard Brownian motion defined on a complete probability basis $(\Omega, \mathcal{F}, \mathcal{F}_t, {\mathbb P})$, where the filtration $ \mathcal{F}_t=\sigma\left\{W(s),\quad s\leq t\right\}$.

\section{Study of the system \eqref{system1}}
\label{sec3}
In this section we prove the existence and uniqueness of the solution of the system \eqref{system1} as well as some uniform estimates. 

\subsection{Well-posedness of the system \eqref{system1}}
For any $\e>0$ we denote by $A^\e$ the matrix 
\begin{equation}\label{Ae}
A^\e: \mathbb{R}^3 \to \mathbb{R}^{3\times 3}, \ \ A^\e(x) = A\left( \dfrac{x}{\e} \right), 
\end{equation}
and by $\alpha^\e$ the operator, 
\begin{equation}\label{alpha}
\alpha^\e: L^2(D) \to L^\infty(D), \ \ \alpha^\e(\eta) (x) = \alpha\left(\dfrac{x}{\e}, \eta(x)\right). 
\end{equation}
Let us show that $\alpha^\e$ is a well defined operator. Given that $\alpha$ is bounded, we need only to show the measurability in $x$ of $\alpha^\e(\eta)$ for any $\eta \in L^2(D)$. For such a function, we consider a sequence $\eta_n \in C_0(D)$ convergent to $\eta$ pointwise in $D$. The function $(y, x) \to \alpha (y, \eta_n(x))$ is a Carath\'{e}odory function, measurable in $y$ and continuous in $x$, so $x\to \alpha\left(\dfrac{x}{\e}, \eta_n(x) \right)$ is measurable, and by the Lipschitz condition of $\alpha$ is pointwise convergent to $\alpha^\e(\eta)$, which shows that $\alpha^\e(\eta)$ is measurable. Moreover we have the following existence and uniqueness result:

\begin{theorem}\label{thexun}
Assume that $u^\e_0\in L^2(D)$ for every $\e>0$, then for each $T>0$, with the possibility of changing the probability space, there exists a unique $\mathcal{F}_t$ - measurable solution of the system \eqref{system1}, $u^\e \in L^\infty(\Omega;C([0, T];L^2(D))\cap L^2(0, T;H_0^1(D)))$ and $v^\e \in L^2(\Omega;C([0, T];L^2(D))$ in the following sense: $\mathbb{P}$ a. s. 
\begin{equation}
\label{weaksole}
\int_D u^\e(t) \phi dx - \int_D u^\e_0 \phi dx + \int_0^t \int_D A^\e\nabla u^\e (s) \nabla \phi dx ds =\int_0^t \int_D \alpha^\e(v^\e) u^\e \phi dx ds + \int_0^t \int_D f(s) \phi dx ds, 
\end{equation}
for every $t\in[0, T]$ and every $\phi \in H_0^1(D)$, and
\begin{equation}
\label{mildsole}
v^\e(t) = v^\e_0 e^{-t/\e} +\frac{1}{\e} \int_0^t u^\e(s) e^{-(t-s)/\e} ds +\frac{\sqrt{Q}}{\sqrt{\e}} \int_0^t e^{-(t-s)/\e} dW(s). 
\end{equation}

Moreover, if the initial conditions $u^\e_0$ are uniformly bounded in $L^2(D)$, then the solutions $u^\e$ satisfies the estimates:
\begin{equation}
\label{est1}
\sup_{\e > 0} \| u^\e \|_{L^\infty (\Omega;L^2(0, T;H_0^1(D)))} \leq C_T, 
\end{equation}
\begin{equation}
\label{est2}
\sup_{\e > 0} \| u^\e \|_{L^\infty (\Omega;C([0, T];L^2(D)))} \leq C_T, 
\end{equation}
and
\begin{equation}
\label{est3}
\sup_{\e > 0} \left\| \dfrac{\partial u^\e}{\partial t} \right\|_{L^\infty (\Omega;L^2(0, T;(H^{-1}(D))))} \leq C_T. 
\end{equation}
Also, if the initial conditions $v^\e_0$ are uniformly bounded in $L^2(D)$ we also have the estimate for $v^\e$:
\begin{equation}
\label{est4}
\sup_{\e > 0} \E \sup_{t\in[0, T]} \| v^\e(t)\|^2_{L^2 (D)} \leq C_T. 
\end{equation}
\end{theorem}
\begin{proof}
We prove the existence of solutions through a Galerkin approximation procedure. We consider $(e_k)_{k\geq 1}$ a sequence of linearly independent elements in $H_0^1(D)\cap L^\infty(D)$ such that $span\{e_k\ | \ k\geq 1\}$ is dense in $H_0^1(D)$. We define the $n$-dimensional space $H_0^1(D)_n$ for every $n>0$ as $span\{e_k\ | \ 1\leq k \leq n\}$ and we denote by $\Pi_n$ the projection operator from $L^2(D)$ onto $H_0^1(D)_n$. 

Let us denote by $w^\e(t)$ the following process
\begin{equation}\label{w}
w^\e(t)=e^{-t/\e}v_0^\e+\frac{\sqrt{Q}}{\sqrt{\e}} \int_0^t e^{-(t-s)/\e} dW(s) \in L^2(\Omega;C([0, T];L^2(D)). 
\end{equation}

Now, in order to prove the existence of solutions, we define the Galerkin approximation 
$$(u^\e_n(t, \omega), z^\e_n(t, \omega) )\in H_0^1(D)_n\times H_0^1(D)_n$$ 
a. s. $\omega\in\Omega$, solution of the following system

\begin{equation}
\label{weaksolen}
\int_D\frac{\partial u^\e_n} {\partial t} (t) \phi dx + \int_D A^\e \nabla u^\e (t) \nabla \phi dx = \int_D \alpha^\e(z^\e_n(t)+w^\e(t)) u_n^\e(t) \phi dx + \int_D f(t) \phi dx, 
\end{equation}
for every $\phi \in H_0^1(D)_n$, $u^\e_n(0, \omega) =\Pi_n u^\e_0$, 
\begin{equation}
\label{mildsolen}
\frac{\partial z^\e_n}{\partial t}(t) = -\frac{1}{\e}(z^\e_n(t)-u^\e_n(t)), \quad z^\e_n(0)=0, 
\end{equation}
where 
\begin{equation}
\label{vne}
z^\e_n(t)=v^\e_n (t)- w^\e(t). 
\end{equation}

Then, we pass to the limit on $(u^\e_n, z^\e_n )$ when $n \to \infty$. 

We write $u_n^\e(\omega, t, x) = \sum_{k=1}^n a^\e_k(\omega, t) e_k(x)$ and 
$z_n^\e(\omega, t, x) = \sum_{k=1}^n b^\e_k(\omega, t) e_k(x)$, and get the following system for the coefficients $a^\e_k$ and $b^\e_k$:
\begin{equation}
\label{weaksolen'}
\left\{
\begin{array}{rll}
& \displaystyle\sum_{k=1}^n\frac{\partial a^\e_k}{\partial t}(\omega, t) \displaystyle\int_D e_k e_l dx+\sum_{k=1}^n a^\e_k(\omega, t) \int_D a_{ij}\left(\dfrac{x}{\e}\right)\dfrac{\partial e_k}{\partial x_j} \dfrac{\partial e_l}{\partial x_i} dx &-\\
&\displaystyle \sum_{k=1}^n\int_D a^\e_k(\omega, t) \alpha^\e\left(w^\e(\omega, t)+\sum_{k=1}^n b_k^\e(\omega, t) e_k\right) e_k e_ldx&= \displaystyle \int_D f(t) e_l dx, \\
\\
& \displaystyle\frac{\partial b^\e_k}{\partial t}(\omega, t)&=-\dfrac{1}{\e}\left( b^\e_k - a_k^\e \right), \ 1\leq k \leq n\\
\\
&a_k^\e(\omega, 0)&=\displaystyle\int_D u^\e_0 e_k dx, \ 1\leq k \leq n\\
\\
&b_k^\e(\omega, 0)&=0, \ 1\leq k \leq n
\end{array}
\right. 
\end{equation}
for each $1\leq l \leq n$. 
We make the following notations: $$b_{ij}=\displaystyle\int_D e_i(x) e_j(x) dx, \ c^\e_{ij}=\displaystyle\ \int_D \sum_{p=1}^{ n}\sum_{ q=1}^{ n}a_{pq}\left(\dfrac{x}{\e}\right)\dfrac{\partial e_i}{\partial x_q} \dfrac{\partial e_j}{\partial x_p} dx, \ f_j(s) = \int_D f(s, x) e_j(x) dx, $$
and
$$(F^\e_n)_{ij}(\omega, t, b_1, ... b_n)= \int_D \alpha^\e\left(w^\e(\omega, t)+\sum_{k=1}^{n} b_k e_k\right) e_i e_jdx$$
and the system is written with these notations as:
\begin{equation}
\label{weaksolen''}
\left\{
\begin{array}{rll}
& \displaystyle\sum_{k=1}^n\frac{\partial a^\e_k}{\partial t} b_{kl}+\sum_{k=1}^n a^\e_k c^\e_{kl}- \sum_{k=1}^n a^\e_k (F^\e_n)_{kl}(b_1^\e, ..., b_n^\e)&= f_l(t), \\
\\
& \displaystyle\frac{\partial b^\e_k}{\partial t}&=-\dfrac{1}{\e}\left( b^\e_k - a_k^\e \right), \ 1\leq k \leq n\\
\\
&a_k^\e(0)&=\displaystyle\int_D u^\e_0 e_k dx, \ 1\leq k \leq n\\
\\
&b_k^\e(0)&=0, \ 1\leq k \leq n
\end{array}
\right. 
\end{equation}
for each $1\leq l \leq n$. 
Given the linearly independence of the sequence $(e_{k})_{k\geq 1}$, the form of the functions $(F^\e_n)_{ij}$ and the Lipschitz condition satisfied by $\alpha$, the system has for every $T>0$ an unique $\mathcal{F}_t$ - measurable solution $(a^\e_k)_{1\leq k \leq n}, \ (b^\e_k)_{1\leq k \leq n} \in C([0, T];L^\infty(\Omega))$, with $(a^\e_k)_{1\leq k \leq n}, \ (b^\e_k)_{1\leq k \leq n} \in W^{1, 2}(0, T)$ a. s. $\omega\in\Omega$. This means that $u_n^\e$ and $z_n^\e=v_n^\e-w^\e$ is a. s. a solution for:
\begin{equation}
\label{weaksolen'''}
\left\{
\begin{array}{rll}
&\displaystyle\int_D\frac{\partial u^\e_n} {\partial t} (t) \phi dx + \int_D A^\e \nabla u_n^\e (t) \nabla \phi dx - \int_D \alpha^\e(z^\e_n(t)+w^\e(t)) u_n^\e(t) \phi dx &= \displaystyle\int_D f(t) \phi dx, \\
&d z_n^\e &=-\dfrac{1}{\e}\left( z^\e_n - u_n^\e \right), \\
\\
&u_n^\e(0)&=\Pi_n u^\e_0, \\
\\
&z_n^\e(0)&=0, 
\end{array}
\right. 
\end{equation}
for every $\phi\in H_0^1(D)_n$. We take $\phi= u_n^\e$ in \eqref{weaksolen'''} to derive that a. e. $\omega\in\Omega$ :

\begin{equation}\nonumber
\begin{split}
\frac{\partial}{\partial t} \|u^\e_n \|^2_{L^2(D)} &\leq \|f(t)\|^2_{L^2(D)} + C \|u^\e_n \|^2_{L^2(D)}\Rightarrow\\
 \|u^\e_n \|^2_{L^2(D)} &\leq e^{Ct} \left( \|f\|_{L^2(0, T;L^2(D))} + \|u^\e_0\|_{L^2(D)}\right), 
\end{split}
\end{equation}
so
\begin{equation}
\label{estune1}
\sup_{n>0}\| u_n^\e\|_{L^\infty(0, T;L^2(D))} \leq C_T(1+\|u^\e_0\|_{L^2(D)}). 
\end{equation}
We also obtain based on the positivity of $A$ that
\begin{equation}\nonumber
\begin{split}
\int_0^T m\|\nabla u^\e_n \|^2_{L^2(D)^3} ds +\frac{1}{2} \| u_n^\e(T)\|^2_{L^2(D)}&\leq \int_0^T \int_D f(t) u_n^\e dxdt + \frac{1}{2}\|u^\e_0 \|^2_{L^2(D)}+\int_0^T C\| u^\e_n \|^2_{L^2(D)} ds\Rightarrow\\
\int_0^T m \|\nabla u^\e_n \|^2_{L^2(D)^3} ds&\leq T \|f\|_{L^2(0, T;L^2(D))}\| u_n^\e\|_{L^\infty(0, T;L^2(D))}+ C_T(1+\|u^\e_0\|_{L^2(D)}), 
\end{split}
\end{equation}
so 
\begin{equation}
\label{estune2}
\sup_{n>0}\| u_n^\e\|_{L^2(0, T;H_0^1(D))} \leq C_T(1+\|u^\e_0\|_{L^2(D)}). 
\end{equation}
The estimates \eqref{estune1} and \eqref{estune2} imply using the first equation of the system \eqref{weaksolen'''} that
\begin{equation}
\label{estune3}
\sup_{n>0}\left\|\dfrac{\partial u_n^\e}{\partial t}\right\|_{L^2(0, T;(H_0^1(D)_n)')} \leq C_T(1+\|u^\e_0\|_{L^2(D)}). 
\end{equation}

This means that the sequence $u_n^\e$ is bounded in $L^2(0, T;H_0^1(D))\cap W^{1, 2}(0, T;H^{-1}(D))$ which is compactly embedded in $L^2(0, T;L^2(D))$ (Theorem 2. 1, page 271 from \cite{temam}) and in $C([0, T], H^{-1}(D))$. Hence, there exists a 
subsequence $u_{n'}^\e$ that converges in distribution in $L^2(0, T;L^2(D))\cap C([0, T], H^{-1}(D))$ to some $u^\e$ which is also a weak limit in $L^2(0, T;H_0^1(D)) \cap W^{1, 2}(0, T;H^{-1}(D))$ and a weak$^*$ limit in $L^\infty(0, T;L^2(D))$. So using Lemma 1. 2, page 260 from \cite{temam} a. s. $\omega\in\Omega$, 
$u^\e \in L^2(0, T;H_0^1(D))\cap C ([0, T];L^2(D))\cap W^{1, 2}(0, T;H^{-1}(D))$. 

We also have from \eqref{weaksolen'''} that
$$z^\e_{n'}(t)=\displaystyle\frac{1}{\e}\int_0^t e^{-(t-s)/\e} u^\e_{n'}(s) ds$$
will converge in distribution to $z^\e(t)=\displaystyle\frac{1}{\e}\int_0^t e^{-(t-s)/\e} u^\e(s) ds$ in $C([0, T];L^2(D))$. Skorokhod representation theorem gives us the existence of another probability space $(\widetilde\Omega, \widetilde{\mathcal{F}}, \widetilde{\mathcal{F}}_t, \widetilde{\mathbb{P}})$ with expectation $\widetilde{\E}$, $\widetilde{W}(t)$ an $L^2(D)$-valued standard Brownian motion on $(\widetilde\Omega, \widetilde{\mathcal{F}}, \widetilde{\mathcal{F}}_t, \widetilde{\mathbb{P}})$ identically distributed as $W(t)$, a subsequence $u^{\e}_{n''}$ and a sequence $\widetilde{u^{\e}}_{n''}$ defined on $\widetilde{\Omega}$, with the same distribution in $L^2(0, T;L^2(D))\cap C([0, T], H^{-1}(D))$, in $L^2(0, T;H_0^1(D)) \cap W^{1, 2}(0, T;H^{-1}(D))$ equipped with the weak topology and in $L^\infty(0, T;L^2(D))$ equipped with the weak$^*$ topology that converges pointwise to an element $\widetilde{u^\e}$ with the same distribution as $u^\e$. We remark that the sequence $\widetilde{{u^{\e}}}_{n''}$ is $\widetilde{\mathcal{F}}_t$ - measurable in $H^{-1}(D)$ and that 
$$\widetilde{z^\e}_{n''}(t)=\displaystyle\frac{1}{\e}\int_0^t e^{-(t-s)/\e} \widetilde{u^\e}_{n''}(s) ds$$
is identically distributed as $z^\e_{n''}(t)$ in $C([0, T];L^2(D))$ and converges pointwise and in distribution to 
 $\widetilde{z^\e}(t)=\displaystyle\frac{1}{\e}\int_0^t e^{-(t-s)/\e} \widetilde{u^\e}(s) ds$. Also the process
$$
\widetilde{w}^\e(t)=e^{-t/\e}v_0^\e+\frac{\sqrt{Q}}{\sqrt{\e}} \int_0^t e^{-(t-s)/\e} d\widetilde{W}(s) \in L^2(\Omega;C([0, T];L^2(D)). 
$$
is identically distributed as $w^\e(t)$. 

We now pass to the limit when $n''\to\infty$ in the first equation of the system \eqref{weaksolen'''} in expected value. We integrate over $[0, t]$ and get:
\begin{equation}\label{eq4}
\begin{split}
&\E \left|\int_0^t\int_D\frac{\partial u^\e_{n''}} {\partial t} \phi dx ds-\int_0^t\int_D\frac{\partial u^\e} {\partial t} \phi dx ds \right| +\E \left| \int_0^t \int_D A^\e \nabla u_{n''}^\e \nabla \phi dxds-\int_0^t \int_D A^\e \nabla u^\e \nabla \phi dxds\right|\\
+&\E\left| - \int_0^t \int_D \alpha^\e(z^\e_{n''}+w^\e) u_{n''}^\e \phi dxds+\int_0^t \int_D \alpha^\e(z^\e+w^\e) u^\e \phi dxds\right| \to 0, 
\end{split}
\end{equation}
when $n''\to\infty$ which gives
\begin{equation}\label{eq5}
\begin{split}
&\widetilde\E \left|\int_0^t\int_D\frac{\partial \widetilde{ u^\e}_{n''}} {\partial t} \phi dx ds-\int_0^t\int_D\frac{ \partial \widetilde{u^\e}} {\partial t} \phi dx ds \right| + \widetilde\E \left| \int_0^t \int_D A^\e \nabla \widetilde{u}^\e_{n''} \nabla \phi dxds-\int_0^t \int_D A^\e \nabla \widetilde{u^\e} \nabla \phi dxds\right|\\
+& \widetilde\E\left| - \int_0^t \int_D \alpha^\e( \widetilde{z^\e}_{n''}+ \widetilde{w^\e}) \widetilde{u^\e}_{n''} \phi dxds+\int_0^t \int_D \alpha^\e(\widetilde{z^\e}+\widetilde{w^\e})\widetilde{ u^\e} \phi dxds\right| \to 0, 
\end{split}
\end{equation}
when $n''\to\infty$. 

In \eqref{eq5} we pass to the limit pointwise in $\widetilde\omega\in\widetilde\Omega$ using the convergences of the sequences $u^\e_{n''}$ and $\dfrac{\partial u^\e_{n''}}{\partial t}$:
$$\lim_{n''\to\infty} \int_0^t \int_D\frac{\partial \widetilde{u^\e}_{n''}} {\partial t} \phi dx ds= \int_0^t\int_D\frac{\partial \widetilde{u^\e}} {\partial t} \phi dxds$$
and
$$\lim_{n''\to\infty} \int_0^t\int_D A^\e \nabla \widetilde{u^\e}_{n''} \nabla \phi dx ds= \int_0^t\int_D \nabla \widetilde{u^\e} \nabla \phi dx ds. $$

Also
\begin{equation}\nonumber
\begin{split}
\left|\int_0^t \int_D \alpha^\e(\widetilde{z^\e}_{n''}+\widetilde{w^\e}) \widetilde{u^\e}_{n''} \phi dxds- 
\int_0^t\int_D \alpha^\e(\widetilde{z^\e}+\widetilde{w^\e}) \widetilde{u^\e} \phi dx \right| &\leq \\
\left| \int_0^t\int_D \alpha^\e(\widetilde{z^\e}_{n''}+\widetilde{w^\e}) (\widetilde{u^\e}_{n''}-\widetilde{u^\e}) \phi dxds \right|
+ \left|\int_0^t\int_D \left(\alpha^\e(\widetilde{z^\e}_{n''}+\widetilde{w^\e}) -\alpha^\e(\widetilde{z^\e}+\widetilde{w^\e})\right) \widetilde{u^\e} \phi dxds\right| &\leq \\
C \left(\int_0^t\int_D ( \widetilde{u^\e}_{n''}- \widetilde{u^\e})^2 dxds\right)^{1/2}+ C \int_0^t\int_D \left| \widetilde{z^\e}_{n''} - \widetilde{z^\e}\right| | \widetilde{u^\e}| |\phi | dxds& \leq \\
C \| \widetilde{ u^\e}_{n''}- \widetilde{u^\e}\|_{L^2(0, T;L^2(D))} + C \int_0^T \| \widetilde{z^\e}_{n''} - \widetilde{z^\e} \|_{L^2(D)} \| \widetilde{ u^\e} \|_{L^2(D)} \| \phi \|_{L^\infty(D)}, 
\end{split}
\end{equation}
so we obtain that a. s. 
$$\lim_{n''\to\infty} \int_0^t\int_D \alpha^\e( \widetilde{z^\e}_{n''}+ \widetilde{w^\e}) \widetilde{u^\e} _{n''}\phi dxds = 
\int_0^t\int_D \alpha^\e( \widetilde{z^\e}+ \widetilde{w^\e}) \widetilde{u^\e} \phi dx ds. $$ 
We use these convergences and \eqref{eq5} to obtain in the limit:
\begin{equation}
\label{weaksolen''''}
\left\{
\begin{array}{rll}
&\displaystyle\int_0^t \int_D\frac{\partial \widetilde{u^\e}} {\partial t} \phi dxds + \int_0^t \int_D A^\e \nabla \widetilde{u^\e}\nabla \phi dxds - \int_0^t \int_D \alpha^\e( \widetilde{z^\e}+ \widetilde{w^\e}) \widetilde{u^\e}\phi dxds &= \displaystyle \int_0^t \displaystyle\int_D f \phi dxds, \\
&d \widetilde{z^\e} &=-\dfrac{1}{\e}\left( \widetilde{z^\e} - \widetilde{ u^\e } \right), \\
\\
& \widetilde{u^\e}(0)&=u^\e_0, \\
\\
& \widetilde{z^\e}(0)&=0, 
\end{array}
\right. 
\end{equation}
pointwise in $\widetilde\omega \in \widetilde\Omega$ for every $\phi\in H_0^1(D)_n$, so by density it is true for any $\phi \in H_0^1(D)$. Now, let 
$ \widetilde{v^\e}:= \widetilde{z^\e}+ \widetilde{w^\e}$, then we deduce that $( \widetilde{u^\e}, \widetilde{v^\e})$ is a solution for our initial system in the sense given by \eqref{weaksole} and \eqref{mildsole}. The solution $( \widetilde{u^\e}, \widetilde{v^\e})$ is $ \widetilde{\mathcal{F}}_t$ - measurable as the limit of the Galerkin approximation $( \widetilde{u^\e}_{n''}, \widetilde{v^\e}_{n''})$ which is $ \widetilde{\mathcal{F}}_t$ - measurable by construction. 
Furthermore, given the uniform estimates for $u^\e_0$ it is easy to obtain from \eqref{estune1}--\eqref{estune3} the estimates \eqref{est1}--\eqref{est3} and \eqref{est4} follows from the uniform bounds for $v^\e_0$. 

Now, we prove the uniqueness. Let us assume that we have two solutions $\{u^\e_1, v^\e_1\}$ and $\{u^\e_2, v^\e_2\}$ for the system. Then, 
\begin{equation}\nonumber
\begin{split}
\int_D (u^\e_2(t) -u^\e_1(t)) \phi dx + \int_0^t \int_D A^\e(\nabla u^\e_2 -\nabla u^\e_1) \nabla \phi dx ds =
\int_0^t \int_D (\alpha^\e(v^\e_2) u^\e_2 - \alpha^\e(v^\e_1) u^\e_1)\phi dx ds, 
\end{split}
\end{equation}
and
\begin{equation}
\nonumber
v^\e_2(t)-v^\e_1(t) = \frac{1}{\e} \int_0^t (u^\e_2(s)-u^\e_1(s)) e^{-(t-s)/\e} ds. 
\end{equation}
we take $\phi = u^\e_2 - u^\e_1$ and we get:
\begin{equation}
\nonumber
\begin{split}
\int_D (u^\e_2(t) -u^\e_1(t))^2 dx + \int_0^t \int_D A^\e(\nabla u^\e_2 -\nabla u^\e_1)^2 dx ds =\\
\int_0^t \int_D \alpha^\e(v^\e_2) (u^\e_2 - u^\e_1)^2dx ds + \int_0^t \int_D (\alpha^\e(v^\e_2)-\alpha^\e(v^\e_1) )u^\e_1 (u^\e_2 - u^\e_1)dx ds\leq \\
c\int_0^t \|u^\e_2 - u^\e_1\|_{L^2(D)}^2ds + c\int_0^t \int_D |v^\e_2-v^\e_1| |u^\e_1| u^\e_2-u^\e_1| dx ds \leq \\
c\int_0^t \|u^\e_2 - u^\e_1\|_{L^2(D)}^2ds + c \int_0^t \|\alpha^\e(v^\e_2)-\alpha^\e(v^\e_1) \|_{L^2(D)} \|u^\e_1 \|_{L^4(D)}\|u^\e_2 - u^\e_1 \|_{L^4(D)}ds \leq \\
c\int_0^t \|u^\e_2 - u^\e_1\|_{L^2(D)}^2ds + c \left(\int_0^t \|v^\e_2-v^\e_1 \|^2_{L^2(D)} \|u^\e_1 \|^2_{L^4(D)}ds \right)^{1/2} \left(\int_0^t \|u^\e_2 - u^\e_1 \|^2_{L^4(D)}ds\right)^{1/2}\leq \\
c\int_0^t \|u^\e_2 - u^\e_1\|_{L^2(D)}^2ds + c \int_0^t \|v^\e_2-v^\e_1 \|^2_{L^2(D)} \|\nabla u^\e_1 \|^2_{L^2(D)^3}ds +\dfrac{m}{2} \int_0^t \|\nabla u^\e_2 - \nabla u^\e_1 \|^2_{L^2(D)^3}ds\leq\\
c\int_0^t \|u^\e_2 - u^\e_1\|_{L^2(D)}^2ds + c \int_0^t \|v^\e_2-v^\e_1 \|^2_{L^2(D)} \|\nabla u^\e_1 \|^2_{L^2(D)^3}ds +\dfrac{m}{2} \int_0^t \|\nabla u^\e_2 - \nabla u^\e_1 \|^2_{L^2(D)^3}ds, 
\end{split}
\end{equation}
where we used H\"{o}lder's inequality, the imbedding of $H_0^1(D)$ into $L^4(D)$ and the Lipschitz condition of $\alpha$. 
\begin{equation}\nonumber
\begin{split}
\|v^\e_2(t)-v^\e_1(t) \|^2_{L^2(D)} \leq c \int_0^t \| u^\e_2(s)-u^\e_1(s)\|^2
_{L^2(D)} e^{-2(t-s)/\e}ds\\
 \leq cT \sup_{s\in [0, t]} \| u^\e_2(s)-u^\e_1(s)\|^2_{L^2(D)}, 
\end{split}
 \end{equation}
so we obtain:
\begin{equation}
\nonumber
\begin{split}
\sup_{s\in[0, t]}\|u^\e_2(t) - u^\e_1(t)\|^2_{L^2(D)} \leq 
c \int_0^t \sup_{r\in [0, s]} \| u^\e_2(r)-u^\e_1(r)\|^2_{L^2(D)} \left( \| \nabla u^\e_1(s) \|^2_{L^2(D)^3}+1\right)ds. 
\end{split}
\end{equation}
We use Gr\"{o}nwall's lemma for the function $\sup_{s\in[0, t]}\|u^\e_2(t) - u^\e_1(t)\|^2_{L^2(D)}$ to obtain that:
$$\sup_{s\in[0, t]}\|u^\e_2(t) - u^\e_1(t)\|^2_{L^2(D)} \leq \|u^\e_2(0) - u^\e_1(0)\|^2_{L^2(D)} e^{c \displaystyle \int_0^t \left(1+\|\nabla u^\e_1 \|^2_{L^2(D)^3}\right)ds}, $$
which gives the uniqueness and this completes the proof. 
\end{proof}

\begin{theorem}
\label{threg}
Assume that the initial conditions $u_0^\e$ are uniformly bounded in $H_0^1(D)$. Then the solution $u^\e \in L^\infty (\Omega;L^2(0, T;H^2(D))) \cap L^\infty (\Omega; C([0, T];H^1_0(D)))$ and satisfies the improved uniform estimates:
\begin{equation}
\label{est2'}
\sup_{\e > 0} \| u^\e \|_{L^\infty (\Omega;C([0, T];H_0^1(D)))} \leq C_T, 
\end{equation}
and
\begin{equation}
\label{est3'}
\sup_{\e > 0} \left\|\dfrac{ \partial u^\e}{\partial t} \right\|_{L^\infty (\Omega;L^2(0, T; L^2(D))} \leq C_T. 
\end{equation}
\begin{proof}
To show these estimates we go back to the Galerkin approximation used to show the existence. In the system \eqref{weaksolen'''} we take $\phi=\dfrac{\partial u_{n}^\e}{\partial t} (t)$ and get
\begin{equation}\nonumber
\displaystyle\int_D\left| \frac{\partial u^\e_n} {\partial t} (t)\right|^2 dx + \int_D A^\e \nabla u_n^\e (t) \nabla\frac{\partial u^\e_n} {\partial t} (t) dx \leq C \left\| \frac{\partial u^\e_n} {\partial t} (t)\right\|_{L^2(D)}\left(\|f(t)\|_{L^2(D)} + \|u_n^\e (t)\|_{L^2(D)} \right). 
\end{equation}
We integrate on $[0, t]$ and use the estimates already obtained for $u^\e_n$ to get:
\begin{equation}\nonumber
\int_0^t \left\| \frac{\partial u^\e_n} {\partial t} (s)\right\|^2_{L^2(D)}ds + m \left\| \nabla u^\e_n(t)\right\|^2_{L^2(D)} \leq M\left\| \nabla u^\e_n(0)\right\|^2_{L^2(D)} + C\int_0^t \left\| \frac{\partial u^\e_n} {\partial t} (s)\right\|_{L^2(D)}, 
\end{equation}
and from here
$$\sup_{\e>0} \sup_{n>0} \int_0^t\left\| \frac{\partial u^\e_n} {\partial t} (s)\right\|^2_{L^2(D)}ds \leq C_T, $$
and
$$\sup_{\e>0} \sup_{n>0} \sup_{t\in[0, T]}\left\| \nabla u^\e_n(t)\right\|^2_{L^2(D)} \leq C_T, $$
which will give us by passing to the limit on the subsequence $u^\e_{n'}$ \eqref{est3'} and 
\begin{equation}
\nonumber
\sup_{\e > 0} \| u^\e \|_{L^\infty (\Omega;L^\infty(0, T;H_0^1(D))))} \leq C_T, 
\end{equation}
We use now the first equation from \eqref{system1} and the regularity theorem for the stationary Stokes equation from \cite{temam} to obtain $u^\e\in L^\infty (\Omega;L^2(0, T;H^2(D)))$. We get \eqref{est2'} by using Lemma 1. 2, Section 1. 4 from \cite{temam}. 
\end{proof}
\end{theorem}

\section{The cell problem}
\label{sec4}
%

In this section we introduce $\chi : Y \to \mathbb{R}^3$ the solution of the cell problem that corresponds to the system \eqref{system1}:
\begin{equation}
\label{cellpr}
\left\{
\begin{array}{rll}
\operatorname{div} \left(A(y) \left( I + \nabla \chi(y)\right)\right) &= 0 &\mbox{ in } Y, \\
\chi & - Y periodic, \\
\end{array}
\right. 
\end{equation}
as well as the solution of the adjoint equation $\chi^*$:
\begin{equation}
\label{cellpr*}
\left\{
\begin{array}{rll}
\operatorname{div} \left(A^*(y) \left( I + \nabla \chi^*(y)\right)\right) &= 0 &\mbox{ in } Y, \\
\chi^* & - Y periodic, \\
\end{array}
\right. 
\end{equation}
where $A^*$ is the adjoint of $A$, $A^*=(a^*_{ij})_{1\leq i, j \leq 3}$, $a^*_{ij} = a_{ji}$ for $1 \leq i, j \leq 3$. It follows that $\chi^{\e}(y)=\chi\left(\dfrac{y}{\e}\right)$ is the solution for the equation:
\begin{equation}
\label{cellpre}
\left\{
\begin{array}{rll}
\operatorname{div} \left(A^\e(y) \left( I + \e\nabla \chi^\e(y)\right)\right) &= 0 &\mbox{ in } \e Y, \\
\chi^\e& - \e Y periodic, \\
\end{array}
\right. 
\end{equation}

We define now the homogenized operator $\o{A}$ as
\begin{equation}
\label{homop}
\o{A}=\int_Y A(y) \left( I + \nabla \chi(y)\right) dy. 
\end{equation}

\section{The fast motion equation}
\label{sec5}

In this section, we present some facts for the invariant measure associated
with (\ref{v}). 
We consider the following problem for fixed $\xi \in L^2(D)$:
\begin{equation}
\label{v}
\left\{
\begin{array}{ll}
dv^\xi &= - (v^\xi-\xi)dt + \sqrt{Q} dW, \\
v^\xi(0) &= \eta. 
\end{array}
\right. 
\end{equation}
This equation admits a unique mild solution $v^\xi(t)\in L^2(\Omega; C([0, T];L^2(D)))$ given by:
\begin{equation}
\label{vxieta}
v^\xi(t) = \eta e^{-t} +\xi(1-e^{-t}) + \int_0^t e^{-(t-s)}\sqrt{Q} dW. 
\end{equation}
When needed to specify the dependence with respect to the initial condition the solution will be denoted by $v^{\xi, \eta}(t)$. The following estimate can be 
derived for $v^{\xi, \eta}(t)$. 
\begin{lemma}
\begin{equation}\label{l}
\E \| v^{\xi, \eta}(t) \|^2_{L^2(D)} \leq 2\left(\|\eta\|^2_{L^2(D)} e^{-2t} + \|\xi\|^2_{L^2(D)} + TrQ \right). 
\end{equation}
\end{lemma}
\begin{proof} It is enough to use the It\^o formula for $\| v^{\xi, \eta}(t) \|^2_{L^2(D)}$. 
\end{proof}
\subsection{The asymptotic behavior of the fast motion equation}
\label{subs21}
Let us define the transition semigroup $P_t^\xi$ associated to the equation \eqref{v}
\begin{equation}
P_t^\xi \Phi (\eta) = \E \Phi(v^{\xi, \eta}(t)), 
\end{equation}
for every $\Phi \in B_b(L^2(D))$ and every $\eta \in L^2(D)$. 
It is easy to verify that $P_t^\xi$ is a Feller semigroup because $\mathbb{P}$ a. s.
\begin{equation}\label{feller}
\|v^{\xi, \eta_1} - v^{\xi, \eta_2}\|^2_{L^2(D)} \leq e^{-2t} \|\eta_1-\eta_2 \|^2_{L^2(D)}. 
\end{equation}
We also denote by $\mu^\xi$ the associated invariant measure on $L^2(D)$. We recall that it is invariant for the semigroup $P_t^\xi$ if
$$\int_{L^2(D)} P_t^\xi \Phi (z) d\mu^\xi(z) = \int_{L^2(D)} \Phi (z) d\mu^\xi(z), $$
for every $\Phi \in B_b(L^2(D))$. 
It is obvious that $v^\xi$ is a stationary gaussian process. The equation \eqref{v} admits a unique ergodic invariant measure $\mu^\xi$ that is strongly mixing and gaussian with mean $\xi$ and covariance operator $Q$. All these results can be found in \cite{DPZ2} or \cite{Cerrai}. 

As a consequence of \eqref{feller} we also have:
\begin{equation}\label{inv}
\left |P_t^\xi \Phi (\eta) - \int_{L^2(D)} \Phi (z) d\mu^\xi(z)\right |\leq c[\Phi] e^{-t}(1+\|\eta\|_{L^2(D)} +\|\xi \|_{L^2(D)}), 
\end{equation}
for any Lipschitz function $\Phi$ defined on $L^2(D)$, where $[\Phi]$ is the Lipschitz constant of $\Phi$. This can be shown as it follows:
\begin{equation}
\begin{split}
P_t^\xi \Phi (\eta) - \int_{L^2(D)} \Phi (z) d\mu^\xi(z) &= \int_{L^2(D)} \left(P_t^\xi \Phi (\eta) -P_t^\xi \Phi (z)\right) d\mu^\xi(z)\\
&= \int_{L^2(D)} \left(\E \Phi (v^{\xi, \eta}(t)) - \E \Phi (v^{\xi, z}(t))\right) d\mu^\xi(z) \\
& \leq \int_{L^2(D)} [\Phi] \E \left\|v^{\xi, \eta}(t) - v^{\xi, z}(t)\right\|_{L^2(D)} d\mu^\xi(z)\\
&\leq \int_{L^2(D)} [\Phi] e^{-t}\E \left\|\eta-z\right\|_{L^2(D)} d\mu^\xi(z)\\
&\leq [\Phi] e^{-t} \left(\|\eta\|_{L^2(D)} + \int_{L^2(D)} \left\|z\right\|_{L^2(D)} d\mu^\xi(z)\right). 
\end{split}
\end{equation}
Now \eqref{inv} follows as a result of the following lemma:
\begin{lemma}
\label{mod}
\begin{equation}
 \int_{L^2(D)} \left\| z \right\|_{L^2(D)} d\mu^\xi(z)\leq c\left(1+\left\|\xi\right\|_{L^2(D)}\right). 
 \end{equation}
 \begin{proof}
\begin{equation}
\begin{split}
\int_{L^2(D)} \|z\|_{L^2(D)}d\mu^\xi(z) &= \int_{L^2(D)} P_t^\xi \|z\|_{L^2(D)}d\mu^\xi(z)\\
& = \int_{L^2(D)} \E \|v^{\xi, z}(t)\|_{L^2(D)}d\mu^\xi(z)\\
&\leq \int_{L^2(D)} c(1 + \|\xi\|_{L^2(D)}+ e^{-t}\|z\|_{L^2(D)})d\mu^\xi(z). 
\end{split}
\end{equation}
We fix now $t>0$ and get the result. 
\end{proof}
\end{lemma}
\begin{remark}
For $\xi, \eta \in L^2(\Omega, \mathcal{F}_{t_0}, L^2(D))$, let $v^{\xi, \eta}$ be the solution of the following system, the equivalent of the system \eqref{v} but with random initial conditions $\eta$ and random parameter $\xi$:
\begin{equation}
\label{vr}
\left\{
\begin{array}{ll}
dv^{\xi, \eta} &= - (v^{\xi, \eta}-\xi)dt + \sqrt{Q} dW, \\
v^{\xi, \eta}(t_0) &= \eta. 
\end{array}
\right. 
\end{equation}
The mild solution for \eqref{vr} $v^{\xi, \eta}(t)\in L^2(\Omega; C([t_0, T];L^2(D)))$ exists and is given by:
\begin{equation}
\label{vxietar}
v^{\xi, \eta}(t) = \eta e^{-(t-t_0)} +\xi(1-e^{-(t-t_0)}) + \int_0^{(t-t_0)} e^{-(t-t_0-s)}\sqrt{Q} dW. 
\end{equation}

The estimates provided by \eqref{l} and \eqref{inv} remains\ valid also in the case when $\xi$ and $\eta$ are random. So for any $\xi, \eta \in L^2(\Omega, \mathcal{F}_{t_0}, L^2(D))$, and a.e. $\omega\in \Omega$ we have:

\begin{equation}\label{l1}
\E \left( \|v^{\xi, \eta}(t)\| ^2_{L^2(D)}|\mathcal{F}_{t_0} \right)\leq 2\left(\|\eta\|^2_{L^2(D)} e^{-2(t-t_0)} + \|\xi\|^2_{L^2(D)}+ Tr Q\right), 
\end{equation}
and

\begin{equation}\label{invr}
\E\left(\left |P_t^{\xi(\omega)} \Phi (\eta(\omega)) - \int_{L^2(D)} \Phi (z) d\mu^{\xi(\omega)}(z)\right |\Big | \mathcal{F}_{t_0}\right) \leq c[\Phi] e^{-(t-t_0)}(1+\|\eta(\omega)\|_{L^2(D)} +\|\xi (\omega)\|_{L^2(D)}), 
\end{equation}
a. e. $\omega \in \Omega$, for any Lipschitz function $\Phi$ defined on $L^2(D)$. 
\end{remark}
The equation \eqref{invr} implies the following Lemma:
\begin{lemma}
\label{key}
Let $\Phi \in C^u([0, T]; L^\infty(\Omega;Lip(L^2(D))))$ be an $\mathcal{F}_t$ - measurable process on $Lip(L^2(D))$, and let $0\leq t_0 <t_0+\delta \leq T$. For $\xi, \eta \in L^2(\Omega, \mathcal{F}_{t_0}, L^2(D))$, let $v^{\xi, \eta}$ be the solution of the system \eqref{vr}. We have:
\begin{equation}\label{estkey}
\begin{split}
&\E \left(\left| \frac{1}{\delta} \int_{t_0}^{t_0+\delta} \Phi(s, v^{\xi, \eta}(s)) ds- \int_{L^2(D)}\Phi (s, z) d\mu^{\xi}(z)\right| \Big | \mathcal{F}_{t_0}\right)\leq\\
 & c\left(1+\|\eta\|_{L^2(D)}+\|\xi\|_{L^2(D)}\right) \left(\frac{\|\Phi\|}{\sqrt{\delta}} +\sqrt{\|\Phi\|[\Phi] (\delta)} \right), 
\end{split}
\end{equation}
where $[\Phi]$ is the modulus of uniform continuity of $\Phi$. 
\end{lemma}
\begin{proof}
We first notice that $\Phi:[t_0, t_0+\delta] \times\Omega \times L^2(D)$ is a Carath\'{e}odory function, so the left hand side is a $\mathcal{F}_{t_0}$ -measurable function on $\Omega$. We can also consider that $\Phi (s,\omega,0)=0$ for all $s\in[t_0,t_0+\delta]$ and a.e. $\omega\in\Omega$ so we have:
\begin{equation}\label{phi1}
\Phi (t,\omega,\eta_1)-\Phi (t,\omega,\eta_2) \leq \|\Phi\| \|\eta_1-\eta_2\|_{L^2(D)},
\end{equation}

\begin{equation}\label{phi2}
\Phi (t_1,\omega,\eta)-\Phi (t_2,\omega,\eta) \leq [\Phi](|t_1-t_2|)\|\eta\|_{L^2(D)}.
\end{equation}

\begin{equation}\label{s}
\begin{split}
&\E \left(\left| \int_{t_0}^{t_0+\delta} \Phi(s, v^{\xi, \eta}(s)) ds- \int_{L^2(D)}\Phi (s, \cdot) d\mu^{\xi}\right|^2   \Big | \mathcal{F}_{t_0}\right) = \\
&\E \left(\left(\int_{t_0}^{t_0+\delta} \Phi(s, v^{\xi, \eta}(s)) ds- \int_{L^2(D)}\Phi (s, \cdot) d\mu^{\xi}\right)^2 \Big | \mathcal{F}_{t_0}\right) =\\
&\E\left( \int_{t_0}^{t_0+\delta} \left( \Phi(s, v^{\xi, \eta}(s))- \int_{L^2(D)}\Phi (s, \cdot) d\mu^{\xi}\right) ds \int_{t_0}^{t_0+\delta} \left( \Phi(r, v^{\xi, \eta}(r))- \int_{L^2(D)}\Phi (r, \cdot) d\mu^{\xi}\right) dr \Big | \mathcal{F}_{t_0}\right) =\\ 
&2\E\left(\int_{t_0}^{t_0+\delta} \left( \Phi(s, v^{\xi, \eta}(s))- \int_{L^2(D)}\Phi (s, \cdot) d\mu^{\xi}\right) \int_s^{t_0+\delta} \left( \Phi(r, v^{\xi, \eta}(r))- \int_{L^2(D)}\Phi (r, \cdot) d\mu^{\xi}\right) dr ds\Big | \mathcal{F}_{t_0}\right) . 
\end{split}
\end{equation}
But, for a. e. $\omega\in\Omega$ and all $s\in [0, T]$:
\begin{equation}\nonumber
\begin{split}
\left| \Phi(s, v^{\xi, \eta}(s))- \int_{L^2(D)}\Phi (s, z) d\mu^{\xi}(z)\right| &\leq c \|\Phi\| \int_{L^2(D)}\left\|v^{\xi, \eta}(s)- z\right\|_{L^2(D)}d\mu^{\xi}(z)\\
&\leq c \|\Phi\| \int_{L^2(D)}\left(\left\|v^{\xi, \eta}(s)\right\|_{L^2(D)}+ \left\|z\right\|_{L^2(D)}\right)d\mu^{\xi}(z), 
\end{split}
\end{equation}
so after using \eqref{l1}
\begin{equation}\label{s'}
\E \left(\left( \Phi(s, v^{\xi, \eta}(s))- \int_{L^2(D)}\Phi (s, \cdot) d\mu^{\xi} \right)^2|\mathcal{F}_{t_0}\right)\leq c\|\Phi\|^2 (1+\|\xi\|^2 +\|\eta\|^2).
\end{equation}
for every $s\in [t_0, t_0+\delta]$ a. e. $\omega\in\Omega$. 
Now, using \eqref{phi2}:
\begin{equation}\nonumber
\begin{split}
 \Phi(r, v^{\xi, \eta}(r))&- \int_{L^2(D)}\Phi (r, \cdot) d\mu^{\xi}= \Phi(s, v^{\xi, \eta}(r))- \int_{L^2(D)}\Phi (s, \cdot) d\mu^{\xi}\\
&+ \Phi(r, v^{\xi, \eta}(r))- \Phi(s, v^{\xi, \eta}(r))+\int_{L^2(D)}\Phi (s, \cdot) d\mu^{\xi}-\int_{L^2(D)}\Phi (r, \cdot) d\mu^{\xi}\\
&\leq \Phi(s, v^{\xi, \eta}(r))- \int_{L^2(D)}\Phi (s, \cdot) d\mu^{\xi} + [\Phi](r-s)\left(\|v^{\xi, \eta}(r)\|_{L^2(D)}+\int_{L^2(D)}\|z\|_{L^2(D)} d\mu^{\xi}(z)\right), 
\end{split}
\end{equation}
a. e. $\omega\in\Omega$ to get after using \eqref{s} and \eqref{s'} :
\begin{equation}\label{eq6'}
\begin{split}
&\E \left( \left| \int_{t_0}^{t_0+\delta} \Phi(s, v^{\xi, \eta}(s)) ds- \int_{L^2(D)}\Phi (s, \cdot) d\mu^{\xi}\right|^2 \Big | \mathcal{F}_{t_0}\right)= \\
&2\int_{t_0}^{t_0+\delta} \E\left(\left( \Phi(s, v^{\xi, \eta}(s))- \int_{L^2(D)}\Phi (s, \cdot) d\mu^{\xi}\right) \int_s^{t_0+\delta} \left( \Phi(r, v^{\xi, \eta}(r))- \int_{L^2(D)}\Phi (r, \cdot) d\mu^{\xi}\right) dr \Big | \mathcal{F}_{t_0}\right)ds\leq\\
&2\int_{t_0}^{t_0+\delta} \E\left(\left( \Phi(s, v^{\xi, \eta}(s))- \int_{L^2(D)}\Phi (s, \cdot) d\mu^{\xi}\right) \int_s^{t_0+\delta} \left( \Phi(s, v^{\xi, \eta}(r))- \int_{L^2(D)}\Phi (s, \cdot) d\mu^{\xi}\right) dr\Big | \mathcal{F}_{t_0}\right)ds +\\
&c \|\Phi\| [\Phi](\delta) \left(1+\|\xi\|+\|\eta\|\right)\delta^{3/2} \left(\int_s^{t_0+\delta} \E \left(\left(1+\|\xi\|_{L^2(D)} +\|v^{\xi, \eta}(r)\|_{L^2(D)}\right)^2 \Big | \mathcal{F}_{t_0}\right)dr\right)^{1/2}\leq\\
&2\int_{t_0}^{t_0+\delta} \E\left(\left( \Phi(s, v^{\xi, \eta}(s))- \int_{L^2(D)}\Phi (s, \cdot) d\mu^{\xi}\right) \int_s^{t_0+\delta} \left( \Phi(s, v^{\xi, \eta}(r))- \int_{L^2(D)}\Phi (s, \cdot) d\mu^{\xi}\right) drds \Big | \mathcal{F}_{t_0}\right) +\\
& c\delta^2 \|\Phi\| [\Phi](\delta)(1+\|\xi\|^2_{L^2(D)} +\|\eta\|^2_{L^2(D)}). 
\end{split}
\end{equation}
But:
\begin{equation}\nonumber\begin{split}
&\E\left(\left( \Phi(s, v^{\xi, \eta}(s))- \int_{L^2(D)}\Phi (s, \cdot) d\mu^{\xi}\right) \int_s^{t_0+\delta} dr \left(\Phi(s, v^{\xi, \eta}(r))- \int_{L^2(D)}\Phi (s, \cdot) d\mu^{\xi}\right)  \Big | \mathcal{F}_{t_0}\right)=\\
&\E\left(\left( \Phi(s, v^{\xi, \eta}(s))- \int_{L^2(D)}\Phi (s, \cdot) d\mu^{\xi}\right) \int_s^{t_0+\delta} dr \E\left( \left(\Phi(s, v^{\xi, \eta}(r))- \int_{L^2(D)}\Phi (s, \cdot) d\mu^{\xi}\right)|\mathcal{F}_s  \right)  \Big | \mathcal{F}_{t_0}\right)=\\
&\E\left(\left( \Phi(s, v^{\xi, \eta}(s))- \int_{L^2(D)}\Phi (s, \cdot) d\mu^{\xi}\right) \int_s^{t_0+\delta} dr \E\left( \left(P_{r-s}^{\xi}\Phi(s, v^{\xi, \eta}(s))- \int_{L^2(D)}\Phi (s, \cdot) d\mu^{\xi}\right)|\mathcal{F}_s  \right)  \Big | \mathcal{F}_{t_0}\right)\leq\\
&c \|\Phi\|(1+\|\xi\|_{L^2(D)}+\|\eta\|_{L^2(D)})\left(\E\left(\int_s^{t_0+\delta}dr \E\left( \left(P_{r-s}^{\xi}\Phi(s, v^{\xi, \eta}(s))- \int_{L^2(D)}\Phi (s, \cdot) d\mu^{\xi}\right)|\mathcal{F}_s  \right)  \right)^2\Big | \mathcal{F}_{t_0}\right)^{1/2}, 
\end{split}\end{equation}
and using \eqref{invr} we have that a. e. $\omega\in\Omega$ and all $s\in[t_0, t_0+\delta]$:
\begin{equation}\nonumber
\begin{split}
\E\left(\left|P_{r-s}^{\xi} \Phi(s, v^{\xi, \eta}(s))- \int_{L^2(D)}\Phi (s, z) d\mu^{\xi}(z)\right| \Big | \mathcal{F}_{s}\right)&\leq c e^{-(r-s) }\|\Phi\| \left( 1+\|\xi\|_{L^2(D)}+\|v^{\xi, \eta}(s)\|_{L^2(D)} \right), 
\end{split}
\end{equation} 
so we get that
\begin{equation}\nonumber
\begin{split}
&\E\left(\left( \Phi(s, v^{\xi, \eta}(s))- \int_{L^2(D)}\Phi (s, z) d\mu^{\xi}(z)\right) \int_s^{t_0+\delta} dr \left(\Phi(s, v^{\xi, \eta}(r))- \int_{L^2(D)}\Phi (s, z) d\mu^{\xi}(z)\right)  \Big | \mathcal{F}_{t_0}\right)\leq\\
& c\|\Phi\|(1+\|\xi\|_{L^2(D)}+\|\eta\|_{L^2(D)})  \left(\int_s^{t_0+\delta} dr e^{-(r-s) }\right) \|\Phi\| \E\left(\left( 1+\|\xi\|_{L^2(D)}+\|v^{\xi, \eta}(s)\|_{L^2(D)} \right)^2\Big | \mathcal{F}_{t_0}\right)^{1/2}\leq\\
& c\|\Phi\|^2(1+\|\xi\|_{L^2(D)}+\|\eta\|_{L^2(D)})^2\left( 1-e^{-(t_0+\delta-s)}\right). 
\end{split}
\end{equation}
The equation \eqref{eq6'} becomes now
\begin{equation}\label{eq6''}
\begin{split}
&\left(\E \left( \left| \frac{1}{\delta} \int_{t_0}^{t_0+\delta} \Phi(s, v^{\xi, \eta}(s)) ds- \int_{L^2(D)}\Phi (s, z) d\mu^{\xi}(z)\right| \Big | \mathcal{F}_{t_0} \right)\right)^2 \leq \\
& c \|\Phi\| [\Phi](\delta)(1+\|\xi\|_{L^2(D)}+\|\eta\|_{L^2(D)})^2+c\frac{1}{\delta^2}\int_{t_0}^{t_0+\delta} \|\Phi\|^2(1+\|\xi\|_{L^2(D)}+\|\eta\|_{L^2(D)})^2 ds\leq\\
& c \|\Phi\| [\Phi](\delta)(1+\|\xi\|_{L^2(D)}+\|\eta\|_{L^2(D)})^2+c\frac{1}{\delta} \|\Phi\|^2(1+\|\xi\|_{L^2(D)}+\|\eta\|_{L^2(D)})^2, 
\end{split}
\end{equation}
which proves the Lemma. 
\end{proof}

\section{Passage to the limit}
\label{sec6}
The main goal of this section is to pass to the limit in the system \eqref{system1} when $\e\to 0$. 
We introduce the following averaged operators:
\begin{equation}
\label{defoalphae}
\o{\alpha^\e}: L^2(D) \to L^\infty(D), \ \ \o{\alpha^\e}(\xi) = \int_{L^2(D)}\alpha^\e(\eta) d\mu^\xi (\eta)
\end{equation}

\begin{equation}
\label{defoalpha}
\o{\alpha}: L^2(D) \to L^\infty(D), \ \ \o{\alpha} (\xi) = \int_{L^2(D)} \left(\int_Y \alpha(y, z)dy\right) d\mu^\xi(z). 
\end{equation}

We remark that $\alpha^\e$ as an operator from $L^2(D)$ to $L^2(D)$ is Lipschitz and $L^2(D)$ is separable, so Pettis Theorem implies that $\alpha^\e: L^2(D) \to L^2(D)$ is measurable. The boundedness of $\alpha^\e$ implies the integrability with respect to the probability measure $\mu^\xi$, so $\o{\alpha^\e}$ is well defined (see Chapter 5, Sections 4 and 5 from \cite{Yosida} for details). The same considerations hold also for the operators $z\in L^2(D) \to \o{\alpha} (z)=\displaystyle\int_Y \alpha(y, z)dy \in L^\infty(D)$, so $\o{\alpha}$ is also well defined. 
Our main result is given by the next theorem. 
\begin{theorem}
\label{thconv1}
Assume the sequence $u^\e_0$ is uniformly bounded in $H_0^1(D))$ and strongly convergent in $L^2(D)$ to some function $u_0$, and $v^\e_0$ is uniformly bounded in $L^2(D)$. Then, there exists $\o{u} \in L^2(0, T;H_0^1(D)))\cap C([0, T]; L^2(D))$ such that $u^\e$ converges in probability to $\o{u}$ in $w\mbox{-}L^2(0, T;H_0^1(D)))\cap C([0, T]; L^2(D))$ and $\o{u}$ is the solution of the following deterministic equation:
\begin{equation}
\label{eqou1}
\left\{
\begin{array}{rll}
\dfrac{\partial \o{u}}{\partial t} &= \operatorname{div} \left( \o{A} \nabla\o{u}\right) + \o{\alpha} (\o{u}) \o{u} + f &\mbox{ in }\ D, \\
\o{u} &=0 &\mbox{ on }\ \partial D, \\
\o{u}(0) & = u_0&\mbox{ in }\ D. 
\end{array}
\right. 
\end{equation}
\end{theorem}

Let us explain the main ideas involved in the proof of this convergence. The uniform bounds for $u^\e$ provided by Theorem \ref{threg} imply that the sequence is tight in $w\mbox{-}L^2(0, T;H^1_0(D))\cap C([0, T]; L^2(D))$, so there exists a limit $\o{u}$ in distribution. We apply after that Skorokhod theorem to get another sequence $\wt{u^\e}$ defined on some probability space $\wt\Omega$, with same distribution as $u^\e$ that converges for a. e. 
$\wt\omega\in \wt\Omega$ to some $\wt{\o{u}}$ in $w\mbox{-}L^2(0, T; H_0^1(D)))\cap C([0, T]; L^2(D))$. We show that $\wt{\o{u}}$ is deterministic and get an equation for it by passing to the limit in expected value in the variational formulation. More precisely, we prove first that:
\begin{equation}
\label{dif}
\begin{split}
\lim_{\e \to 0} &\ \E \left | \int_0^T \int_D \left( \alpha^\e(v^\e(t)) u^\e(t) - \o{\alpha}(\o{u}(t)) \o{u}(t) \right)\phi^\e \psi(t) dx dt \right | = 0, 
\end{split}
\end{equation}
for a particular sequence $\phi^\e \in H^1_0(D)$ and any $\psi \in C[0, T]$. 
We rewrite it as:
$$\int_0^T \int_D \left( \alpha^\e(v^\e(t)) u^\e(t) - \o{\alpha}(\o{u}(t)) \o{u}(t) \right)\phi^\e \psi(t) dx dt = S^\e_1 + S^\e_2 +S^\e_3, $$
where
\begin{equation}\label{S1}
S^\e_1 = \int_0^T \int_D \left(\alpha^\e(v^\e(t)) - \o{\alpha^\e}(u^\e(t))\right)u^\e(t)\phi^\e \psi(t) dxdt, 
\end{equation}
\begin{equation}\nonumber
S^\e_2 = \int_0^T \int_D \left( \o{\alpha^\e}(u^\e(t))u^\e(t) - \o{\alpha^{\e}}(\o{u}(t))\o{u}(t) \right)\phi^\e \psi(t)dx dt, 
\end{equation}
and
\begin{equation}\nonumber
S^\e_3 = \int_0^T \int_D \left(\o{\alpha^{\e}}(\o{u}(t))\o{u}(t) - \o{\alpha}(\o{u}(t))\o{u}(t)\right)\phi^\e \psi(t) dxdt. 
\end{equation}
This convergence requires two steps. The first step is performed in Subsection \ref{subs43} where we prove the convergence to $0$ for $S^\e_1$. This is done by proving the more general result \eqref{S1} where the equation satisfied by $u^\e$ is not important. The idea is to approximate $u^\e$ and $\phi^\e$ by step functions in time and use Lemma \ref{key} on each piece. In Subsection \ref{subs44} we do the second step, the convergence to $0$ of $S^\e_2$. In Subsection \ref{subs45} we show the convergence to $0$ of $S^\e_3$, which is showed in Lemma \ref{lemmaconv1'}. 

The sequence $\wt{u^\e}$ given by Skorokhod theorem converges a. s. to $\wt{\o{u}}$ weakly in $L^2(0, T;H^1_0(D))$ and strongly in $C([0, T];L^2(D))$ so 
\begin{equation}\label{dif1}
\lim_{\e \to 0} \left | \int_0^T \int_D \left(\wt{u^\e} (t)-\wt{\o{u}} (t)\right) \phi^\e \psi'(t) dx dt - \int_0^T \int_D \left( A^\e \nabla \wt{u^\e} - \o{A}\nabla\wt{\o{u}} \right) \nabla \phi \psi(t) dx dt\right | = 0, \quad a. s. \end{equation}
The equations \eqref{dif} and \eqref{dif1} imply that $\wt{\o{u}}$ satisfies almost surely the variational formulation associated with \eqref{eqou1}, so $\wt{\o{u}}$ and $\o{u}$ are deterministic and as a consequence the convergence of the sequence $u^\e$ to $\o{u}$ will be in probability. 
Before proceeding with the proof of Theorem \ref{thconv1}, let us first study system \eqref{eqou1}. 

\subsection{Well-possedness for the averaged equation \eqref{eqou1}}
\label{subs41}
\begin{theorem}\label{thexunou}
Assume $f\in L^2(0, T;L^2(D))$ and $\o{\alpha} \in Lip_b(\mathbb{R})$. Then, for any $u_0 \in L^2(D)$ the system \eqref{eqou1} admits a unique solution $\o{u} \in C([0, T];L^2(D))\cap L^2(0, T;H_0^1(D)))$ with $\dfrac{\partial \o{u}}{\partial t} \in L^2(0, T;H^{-1}(D))$ in the following sense:
\begin{equation}
\label{weaksolou}
\int_D \o{u}(t) \phi dx - \int_D u_0 \phi dx + \int_0^t \int_D \o{A}\nabla \o{u} (s) \nabla \phi dx ds =\int_0^t \int_D \o{\alpha}(\o{u}) \o{u} \phi dx ds + \int_0^t \int_D f(s) \phi dx ds, 
\end{equation}
for every $t\in[0, T]$ and every $\phi \in H_0^1(D)$. Moreover, if the initial condition $u_0 \in H_0^1(D)$, then $\o{u}$ has the improved regularity, $\o{u} \in L^2(0, T; H^2(D)) \cap L^\infty(0, T;H_0^1(D))$ and $\dfrac{ \partial \o{u}}{\partial t} \in L^2(0, T; L^2(D))$. 
\begin{proof}
The proof of existence of solutions is similar to the proof of system \eqref{system1}, using a Galerkin approximation procedure. The finite dimensional approximation $\o{u}_n$, defined as in Theorem \ref{thexun} will solve
\begin{equation}
\label{weaksoloun}
\int_D\frac{\partial \o{u}_n} {\partial t} (t) \phi dx + \int_D \o{A}\nabla \o{u}_n (t) \nabla \phi dx = \int_D \o{\alpha}(\o{u}_n) \o{u}_n \phi dx + \int_D f(t) \phi dx, 
\end{equation}
for every $\phi \in C([0, T], H_0^1(D))_n)$, and $\o{u}_n(0) =\Pi_n u_0$. We take $\phi = \o{u}_n(t)$, and get:
\begin{equation*}
\begin{split}
&\int_D\frac{\partial \o{u}_n} {\partial t} (t) \o{u}_n(t) dx + \int_D m\|\nabla \o{u}_n (t)\|^2 dx \leq c\int_D |\o{u}_n(t) |^2dx + \int_D f(t) \o{u}_n(t)dx \Rightarrow \\
&\dfrac{ \partial} {\partial t} \|\o{u}_n(t) \|^2_{L^2(D)} \leq \|f(t) \|^2_{L^2(D)} +c \|\o{u}_n(t) \|^2_{L^2(D)} \Rightarrow \\
&\|\o{u}_n(t) \|^2_{L^2(D)} \leq c +c\int_0^t \|\o{u}_n(s) \|^2_{L^2(D)} ds. 
\end{split}
\end{equation*}
We use Gr\"{o}nwall's lemma and get:
\begin{equation}\label{estunou1}
\sup _{n > 0}\| \o{u}_n \|_{C([0, T]; L^2(D)} \leq C_T, 
\end{equation}
and from here we also obtain
\begin{equation}\label{estunou2}
\sup _{n > 0}\|\nabla \o{u}_n \|_{L^2(0, T;L^2(D)^{3})} \leq C_T, 
\end{equation}
and
\begin{equation}\label{estunou3}
\sup _{n > 0}\left\| \dfrac{ \partial \o{u}_n} {\partial t} \right\|_{L^2(0, T;H^{-1}(D)} \leq C_T. 
\end{equation}
So there exists a subsequence $\o{u}_{n'}$ and a function $\o{u} \in L^\infty(0, T; L^2(D)) \cap L^2(0, T;H_0^1(D))$ such that $\o{u}_{n'}$ converges weakly star in $L^\infty(0, T; L^2(D))$ and weakly to $L^2(0, T;H_0^1(D))$ to $\o{u}$ and also $\dfrac{\partial \o{u}_{n'}}{\partial t} $ converges to $\dfrac{\partial \o{u}}{\partial t} $ weakly in $L^2(0, T;H^{-1}(D))$. We apply again now Theorem 2. 1, page 271 and Lemma 1. 2 page 260 from \cite{temam} to obtain that $\o{u}_{n'}$ converges strongly in $L^2(0, T;L^2(D))$ and in $C([0, T];L^2(D))$ to $\o{u}$. We then pass to the limit and obtain that $\o{u}$ is a weak solution for \eqref{eqou1}. 

Now, to show uniqueness we assume to have two solutions $\o{u}_1$ and $\o{u}_2$ in $ C([0, T];L^2(D))\cap L^2(0, T;H_0^1(D))$ and substract the variational formulations. We get:
\begin{equation}\nonumber
\begin{split}
\int_D (\o{u}_2(t) -\o{u}_1(t)) \phi dx + \int_0^t \int_D \o{A}(\nabla \o{u}_2 -\nabla \o{u}_1) \nabla \phi dx ds =
\int_0^t \int_D (\o{\alpha}(\o{u}_2) \o{u}_2- \o{\alpha}(\o{u}_1) \o{u}_1)\phi dx ds. 
\end{split}
\end{equation}
We take $\phi = \o{u}_2 - \o{u}_1$ and write
\begin{equation*}
\begin{split}
(\o{\alpha}(\o{u}_2) \o{u}_2 - \o{\alpha}(\o{u}_1) \o{u}_1) (\o{u}_2 - \o{u}_1) &= \o{\alpha}(\o{u}_2) (\o{u}_2 - \o{u}_1)^2 + \o{u}_1 (\o{\alpha}(\o{u}_2) - \o{\alpha}(\o{u}_1))(\o{u}_2 - \o{u}_1)\\
&\leq C (\o{u}_2 - \o{u}_1)^2 + C|u_1||\o\alpha(u_2)-\o\alpha(u_1)| |u_2-u_1|. 
\end{split}
\end{equation*}
We get
\begin{equation*}
\begin{split}
 &\| \o{u}_2(t) -\o{u}_1(t) \|^2_{L^2(D)}+ \int_0^t m\| \nabla \o{u}_2 - \nabla \o{u}_1\|^2_{L^2(D)^{3}}ds\\
& \leq C \int_0^t \|\o{u}_2 - \o{u}_1\|_{L^2(D)}^2 ds+C \int_0^t \|\o{u}_1\|_{L^4(D)} \|\o{\alpha}(\o{u}_2) - \o{\alpha}(\o{u}_1)\|_{L^2(D)} \|\o{u}_2 - \o{u}_1\|_{L^4(D)} ds\\
& \leq C \int_0^t \|\o{u}_2 - \o{u}_1\|_{L^2(D)}^2 ds+C(\e)\int_0^t \|\o{u}_1 \|_{L^4(D)}^2 \| \o{\alpha}(\o{u}_2) - \o{\alpha}(\o{u}_1)\|^2_{L^2(D)} ds+\e \int_0^t \|\o{u}_2 - \o{u}_1\|^2_{L^4(D)} ds\\
&\leq C \int_0^t \|\o{u}_2 - \o{u}_1\|_{L^2(D)}^2 ds +C(\e)\int_0^t \|\nabla \o{u}_1 \|_{L^2(D)^3}^2 \| \o{u}_2 - \o{u}_1\|^2_{L^2(D)} ds+\e \int_0^t \|\nabla \o{u}_2 -\nabla \o{u}_1\|^2_{L^2(D)^3} ds. 
\end{split}
\end{equation*}
after using H\"{o}lder's inequality and Sobolev imbedding theorem. We obtain for a convenient choice of $\e$
\begin{equation*}
\begin{split}
\| \o{u}_2(t) -\o{u}_1(t) \|^2_{L^2(D)} \leq c(\e) \int_0^t \| \o{u}_2(s) -\o{u}_1(s) \|^2_{L^2(D)} \left(1+ \|\nabla \o{u}_1 \|_{L^2(D)^3}^2\right). 
\end{split}
\end{equation*}
We get uniqueness from here by applying Gr\"{o}nwall's lemma. 

Let us now assume that the initial condition $u_0 \in H_0^1(D)$. We use the equation \eqref{weaksoloun} with $\phi = \dfrac{\partial \o{u}_n}{\partial t}$:
\begin{equation}
\label{weaksoloun'}
\int_D\left(\frac{\partial \o{u}_n} {\partial t} (t)\right)^2 dx + \int_D \o{A}\nabla \o{u}_n (t) \nabla \dfrac{\partial \o{u}_n}{\partial t} (t) dx = \int_D \o{\alpha}(\o{u}_n(t)) \o{u}_n(t) \dfrac{\partial \o{u}_n}{\partial t} (t)dx+ \int_D f(t) \dfrac{\partial \o{u}_n}{\partial t} (t)dx, 
\end{equation}
we integrate it over $[0, T]$, and use H\"{o}lder's inequality:
$$ \left\| \frac{\partial \o{u}_n} {\partial t}\right\|^2_{L^2(0, T; L^2(D))} + m\| \nabla \o{u}_n (T)\|^2_{L^2(D)^{3}} - M\| \nabla \o{u}_n (0)\|^2_{L^2(D)^{3}} \leq C \left\| \frac{\partial \o{u}_n} {\partial t}\right\|_{L^2(0, T; L^2(D))}, $$
which will imply that $\dfrac{\partial \o{u}_n} {\partial t} \in L^2(0, T; L^2(D))$ uniformly bounded and $\nabla u_n \in L^\infty(0, T;L^2(D)^3)$ uniformly bounded. Regularity theorem for the stationary Stokes equation implies that $\Delta \o{u}_n \in L^2(0, T; L^2(D) )$ and is uniformly bounded and $\o{u}_n \in L^2(0, T; H^2(D))$ and is uniformly bounded. We deduce by passing to the limit that $\o{u} \in L^2(0, T; H^2(D))$, $ \dfrac{\partial \o{u}} {\partial t}\in L^2(0, T; L^2(D)) $ and $\o{u} \in C([0, T]; H_0^1(D))$. 

\end{proof}
\end{theorem}

\subsection{Convergence of $S^\e_1$}\label{subs42}

\begin{lemma}
\label{converg}
Assume that $u^\e$ is a sequence of $\mathcal{F}_t$ - measurable processes in $L^2(D)$, uniformly bounded in $L^\infty(\Omega, W^{1, 2}(0, T;L^{2}(D)))$, $\phi^\e$ a sequence of $\mathcal{F}_t$ - measurable processes in $L^2(D)$, such that $\phi^\e \in L^\infty(\Omega;C^u([0, T] ;L^2(D)))$ uniformly bounded and equiuniform continuous with respect to $\e >0$ and $\omega\in\Omega$. Let the sequence $v^\e$ satisfy the equation
\begin{equation} \label{ve}
\left\{
\begin{array}{rll}
d v^\e(t, x) &= -\dfrac{1}{\e} (v^\e(t, x)-u^\e(t, x))dt + \sqrt {\dfrac{Q}{\e}} dW(t, x) &\mbox{ in }\ [0, T] \times D, \\
v^\e(0, x) & = v_0^\e(x)&\mbox{ in }\ D, 
\end{array}\right. 
\end{equation}
with the sequence $v_0^\e$ uniformly bounded in $L^2(D)$. Then we have that:
\begin{equation}
\lim_{\e\to 0} \E \left| \int_0^T \int_D \left(\alpha^\e(v^\e(t)) - \o{\alpha}^\e (u^\e(t))\right) \phi^\e(t) dx dt\right| =0. 
\end{equation}
\end{lemma}
\begin{proof}
Fix $n^\e$ a positive integer and let $\delta^\e =\dfrac{T}{n^\e}$. We define $\wt{u}^\e$ as the piecewise constant function:
\begin{equation}\label{wtphie}
\wt{u}^\e(t)=u^\e(k\delta^\e) \ \mbox{ for } t\in[k\delta^\e, (k+1)\delta^\e). 
\end{equation}
We define also the sequence $\wt{v}{^\e}$ as the solution of:
\begin{equation} \label{wtve}
\left\{
\begin{array}{rll}
d \wt{v}^\e(t, x) &= -\dfrac{1}{\e} (\wt{v}^\e(t, x)-\wt{u}^\e(t, x))dt + \sqrt {\dfrac{Q}{\e}} dW(t, x) &\mbox{ in }\ [0, T] \times D, \\
\wt{v}^\e(0, x) & = v_0^\e(x)&\mbox{ in }\ D. 
\end{array}
\right. 
\end{equation}
A simple calculation shows that the sequence $u^\e$ is H\"{o}lder continuous, uniformly in $\e$ and $\omega$:
\begin{equation}\nonumber
\begin{split}
u^\e(t) - u^\e(s) &=\int_s^t \dfrac{\partial u^\e}{\partial t}(r) dr \Rightarrow\\
\|u^\e(t) - u^\e(s)\|_{L^2(D)}&\leq (t-s)^{\frac{1}{2}}\left(\int_0^T \left\| \dfrac{\partial u^\e}{\partial t}(r)\right\|^2_{L^{2}(D)} dr \right)^{\frac{1}{2}}\leq C(t-s)^{\frac{1}{2}}. 
\end{split}
\end{equation}
This implies that:
\begin{equation}\label{difue}
\lim_{\delta^\e \to 0} \| \wt{u}^\e - u^\e \|_{L^\infty(0, T; L^2(D))} =0, 
\end{equation}
uniformly in $\e$ and $\omega$. 
From \eqref{ve} and \eqref{wtve} we get that $\wt{v}^\e(t) - v^\e(t) = \dfrac{1}{\e}\displaystyle\int_0^t e^{\frac{-(t-s)}{\e}} \left(\wt{u}^\e(s) - u^\e(s)\right)ds$, so we also have that
\begin{equation}\label{difve}
\lim_{\delta^\e \to 0} \| \wt{v}^\e - v^\e \|_{L^\infty(0, T; L^2(D))} =0, 
\end{equation}
uniformly in $\e$ and $\omega$. 

Now
\begin{equation}\nonumber
\begin{split}
\int_0^T \int_D \left(\alpha^\e(v^\e(t)) - \o{\alpha}^\e (u^\e(t))\right) \phi^\e(t) dx dt -\int_0^T \int_D \left(\alpha^\e(\wt{v}^\e(t)) - \o{\alpha}^\e (\wt{u}^\e(t))\right) \phi^\e(t) dx dt=\\
\int_0^T \int_D \phi^\e(t)\left(\alpha^\e(v^\e(t))-\alpha^\e(\wt{v}^\e(t)) \right) dxdt + \int_0^T \int_D \phi^\e(t)\left(\o{\alpha}^\e(\wt{u}^\e(t))-\o{\alpha}^\e(u^\e(t)) \right) dxdt, \end{split}
\end{equation}
But:
\begin{equation}\nonumber
\begin{split}
&\int_0^T \int_D \phi^\e(t)\left(\alpha^\e(v^\e(t))-\alpha^\e(\wt{v}^\e(t)) \right) dxdt\leq\\
&\|\phi^\e\|_{L^\infty(\Omega; C([0, T]; L^2(D)))}\int_0^T \left(\int_D |\alpha^\e(v^\e(t)) - \alpha^\e (\wt{v}^\e(t)|^{2}dx \right)^{1/2}\leq\\
&\|\phi^\e\|_{L^\infty(\Omega; C([0, T]; L^2(D)))} \int_0^T \left(\int_D [\alpha]^2 \left|v^\e(t) - \wt{v}^\e(t)\right|^2 dx \right)^{1/2}\leq\\
&CT \|\phi^\e\|_{L^\infty(\Omega; C([0, T]; L^2(D)))} [\alpha]\| \wt{v}^\e - v^\e \|_{L^\infty(0, T; L^2(D))}, 
\end{split}
\end{equation}
and similarly 
\begin{equation}\nonumber
\begin{split}
&\int_0^T \int_D \phi^\e(t)\left(\o{\alpha}^\e(\wt{u}^\e(t))-\o{\alpha}^\e(u^\e(t)) \right) dxdt\leq\\
&C T\|\phi^\e\|_{C([0, T]; L^\infty(\Omega;L^2(D)))} [\alpha] \| \wt{u}^\e - u^\e \|_{L^\infty(0, T; L^2(D))}, 
\end{split}
\end{equation}
which will imply based on \eqref{difue} and \eqref{difve} that
\begin{equation}\label{difalpha}
\lim_{\delta^\e \to 0}\E\left| \int_0^T \int_D \left(\alpha^\e(v^\e(t)) - \o{\alpha}^\e (u^\e(t))\right) \phi^\e(t) dx dt -\int_0^T \int_D \left(\alpha^\e(\wt{v}^\e(t)) - \o{\alpha}^\e (\wt{u}^\e(t))\right) \phi^\e(t) dx dt\right| =0, 
\end{equation}
uniformly in $\e$. 

Let us study now the term $\displaystyle \int_0^T \int_D \left(\alpha^\e(\wt{v}^\e(t)) - \o{\alpha}^\e (\wt{u}^\e(t))\right) \phi^\e(t) dx dt$. 
\begin{equation}
\begin{split}\label{eq3}
\int_0^T \int_D \left(\alpha^\e(\wt{v}^\e(t)) - \o{\alpha}^\e (\wt{u}^\e(t))\right) \phi^\e(t) dx dt &=\sum_{k=0}^{n^\e-1} \int_{k \delta^\e}^{(k+1)\delta^\e} \int_D\left(\alpha^\e(\wt{v}^\e(t)) - \o{\alpha}^\e (\wt{u}^\e(t))\right) \phi^\e(t) dx dt. \end{split}
\end{equation}

The process defined by 
\begin{equation}\label{deffe1}
F^\e(s, \eta)= \displaystyle\int_D \alpha^\e(\eta) \phi^\e \left(\e s\right)dx
\end{equation}
belongs to $C^u([0, T / \e ] ;Lip(L^2(D)))$, 
with 
$$| F^\e (s, 0) | \leq |\alpha| \|\phi^\e\|_{C([0, T]; L^2(D))}, $$
$$ [F^\e(s, \cdot)] \leq [\alpha] \left\| \phi^\e\right\|_{C([0, T];L^2(D))}, $$ 
so
$$ \| F^\e (s)\|_{Lip(L^2(D))} \leq (|\alpha|+[\alpha]) \|\phi^\e\|_{C([0, T]; L^2(D))}$$
and 
$$[F^\e] (r) \leq(|\alpha|+[\alpha]) [\phi^\e]_{C^u([0, T]; L^2(D))}(\e r), $$ so we can apply Lemma \ref{key} on the interval $[k \delta_\e / \e, (k+1) \delta_\e / \e]$ for $\xi = u^\e(k\delta^\e)$ and $\eta=\wt{v}^\e(k \delta^\e)$ to the sequence $ F^\e$:

\begin{equation}\label{eq1}
\begin{split}
&\E \left(\left| \frac{\e}{\delta^\e} \int_{k\delta^\e / \e}^{(k+1)\delta^\e / \e} F^\e(s, v^{u^\e(k\delta^\e), \wt{v}^\e(k \delta^\e)}(s)) ds- \int_{L^2(D)} F^\e (s, z) d\mu^{u^\e(k\delta^\e)}(z)\right| \Big | \mathcal{F}_{k\delta^\e}\right)\leq \\
&c \left(1+\|\wt{v}^\e(k \delta^\e)\|_{L^2(D)}+\|u^\e(k\delta^\e)\|_{L^2(D)}\right)\left( \frac{\sqrt{\e}\|F^\e\| }{\sqrt{\delta_\e}} + \sqrt{\|F^\e\| [F^\e]( \delta_\e / \e)} \right)\leq \\
&C \left(1+\|\wt{v}^\e(k \delta^\e)\|_{L^2(D)}+\|u^\e(k\delta^\e)\|_{L^2(D)}\right)\left( \frac{\sqrt{\e}\|\phi^\e\|}{\sqrt{\delta_\e}} + \sqrt{\|\phi^\e\|\left[ \phi^\e \right](\delta_\e )} \right). \\
\end{split}
\end{equation}

But by a change of variables $\wt{v}^\e\left(\e t\right)$ is a solution for the equation \eqref{vr} on the interval $[k\delta^\e / \e, (k+1)\delta^\e / \e]$ with $\xi = u^\e(k \delta^\e)$ and $\eta = \wt{v}^\e(k \delta^\e)$, so 
$$ v^{u^\e(k\delta^\e), \wt{v}^\e(k \delta^\e)}(s) = \wt{v}^\e\left(\e s\right). $$
Also using formula \eqref{l1}:
\begin{equation}\nonumber
\begin{split}
\E\left(\| \wt{v}^\e((k+1)\delta^\e)\|^2_{L^2(D)}|\mathcal{F}_{k\delta^\e}\right) \leq &c\left(\| \wt{v}^\e(k\delta^\e)\|^2_{L^2(D)} e^{-2\delta^\e / \e} + \| u^\e(k\delta^\e)\|^2_{L^2(D)}+1\right)\Rightarrow\\
\| \wt{v}^\e((k+1)\delta^\e)\|^2_{L^2(\Omega, L^2(D))}\leq & c \left(\| \wt{v}^\e(k\delta^\e)\|^2_{L^2(\Omega, L^2(D))} e^{-2\delta^\e/\e} + \| u^\e\|^2_{L^2(\Omega, C([0, T];L^2(D)))}+1\right), 
\end{split}
\end{equation}
and we obtain by induction that:
\begin{equation}\nonumber
\| \wt{v}^\e(k\delta^\e)\|^2_{L^2(\Omega, L^2(D))} \leq c^ke^{-2k\delta^\e/\e} \| \wt{v}^\e(0\|^2_{L^2(\Omega, L^2(D))} +\left(\sum_{i=1}^k c^i e^{-2i\delta^\e/\e}\right) \left(\| u^\e\|^2_{L^2(\Omega, C([0, T];L^2(D)))}+1\right), 
\end{equation}
so for $\e/\delta^\e$ small enough we get the estimate:
\begin{equation}\label{estweve}
\| \wt{v}^\e(k\delta^\e)\|^2_{L^2(\Omega, L^2(D))} \leq C \left(\| u^\e\|^2_{L^2(\Omega, C([0, T];L^2(D)))}+1\right)\ ,\ \forall k >0. 
\end{equation}
The equation \eqref{eq1} now becomes:
\begin{equation}\label{eq2}
\begin{split}
&\E \left| \frac{\e}{\delta^\e} \int_{k\delta^\e / \e}^{(k+1)\delta^\e / \e} F^\e\left(s, \wt{v}^\e\left(\e s\right)\right) ds- \int_{L^2(D)}F^\e \left(s, z\right) d\mu^{u^\e(k\delta^\e)}(z)\right| =\\
&\E \left| \frac{1}{\delta^\e} \int_{k\delta^\e}^{(k+1)\delta^\e} F^\e(\frac{ s}{\e}, \wt{v}^\e\left(s\right)) ds- \int_{L^2(D)}F^\e (\frac{ s}{\e}, z) d\mu^{u^\e(k\delta^\e)}(z)\right|\leq \\
&C \left(1+\| u^\e\|^2_{L^2(\Omega, C([0, T];L^2(D)))}\right)\left( \frac{\sqrt{\e}\|\phi^\e\|}{\sqrt{\delta_\e}} + \sqrt{\|\phi^\e\|\left[\phi^\e \right](\delta_\e )} \right). 
\end{split}
\end{equation}
If we sum over all $0\leq k \leq n^\e-1$ and go back to the equation \eqref{eq3} we obtain that
\begin{equation}
\begin{split} 
&\E \left| \int_0^T \int_D \left(\alpha^\e(\wt{v}^\e(t)) - \o{\alpha}^\e (\wt{u}^\e(t))\right) \phi^\e(t) dx dt\right| \\
\leq & C \left(1+\|u^\e\|_{C([0, T]; L^\infty(\Omega, L^2(D)))}\right)\left( \frac{\sqrt{\e}\|\phi^\e\| }{\sqrt{\delta_\e}} + \sqrt{\|\phi^\e\|\left[ \phi^\e\right](\delta_\e )} \right). 
\end{split} 
\end{equation}
If we choose now $n^\e = T/ \sqrt{\e}$ use the equiuniform continuity of $\phi^\e$ and the convergences given by \eqref{difalpha} we obtain that 
$$\lim_{\e\to 0}\E \left| \int_0^T \int_D \left(\alpha^\e(v^\e(t)) - \o{\alpha}^\e (u^\e(t))\right) \phi^\e(t) dx dt\right| =0, $$
which proves the Lemma. 
\end{proof}
The convergence to $0$ of $S_1^\e$ is an imediate consequence:
\begin{lemma}\label{lemmaconv1}
If $\phi^\e$ is a sequence uniformly bounded in $H_0^1(D)$ and $\psi\in C[0, T]$ then:
\begin{equation}
\label{convre}
\begin{split}
\lim_{\e\to 0} \E\left| \int_0^T \int_D \left(\alpha^\e(v^\e(t)) - \o{\alpha^\e}(u^\e(t))\right)u^\e(t)\phi^\e \psi(t) dxdt\right| =& 0. 
\end{split}
\end{equation}
\begin{proof}
As $u^\e$ is uniformly bounded in $L^\infty(\Omega, C([0, T]; H_0^1(D))) \cap L^\infty(\Omega, W^{1, 2}(0, T; L^2(D)))$ and $\Psi\in C[0, T]$, then the sequence $u^\e \phi^\e \psi $ is uniformly bounded and equiuniformly continuous in $C([0, T] ; L^\infty(\Omega;L^2(D)))$, so we can apply the previous Lemma. 
\end{proof}
\end{lemma}
\subsection{Convergence of $S^\e_2$}\label{subs43}
\begin{lemma}\label{lemmaconv1''}
Assume $u^\e$ is a sequence uniformly bounded in $L^\infty(\Omega, C([0, T], H_0^1(D)))$ that converges in distribution to $\o{u}$ in $C([0, T], L^2(D)))$. Then, for any sequence $\phi^\e$ uniformly bounded in $H_0^1(D)$ and $\psi\in C[0, T]$ we have:
\begin{equation}\label{convre''}
\lim_{\e\to 0}\E \left|\int_0^T\int_D(\o{\alpha^{\e}}(u^\e(t))u^\e(t) - \o{\alpha^{\e}}(\o{u}(t))\o{u}(t)) \phi^\e\psi(t) dxdt \right| =0. 
\end{equation}
\end{lemma}
\begin{proof}
We compute:
\begin{equation}\begin{split}
(\o{\alpha^{\e}}(u^\e(t))u^\e(t) - \o{\alpha^{\e}}(\o{u}(t))\o{u}(t)) \phi^\e\psi(t)=&\o{\alpha^{\e}}(u^\e(t))(u^\e(t) - \o{u}(t)) \phi^\e\psi(t)\\
+ &u^\e(t)(\o{\alpha^{\e}}(u^\e(t))- \o{\alpha^{\e}}(\o{u}(t))) \phi^\e\psi(t), 
\end{split}
\end{equation}
so
\begin{equation}\begin{split}
&\E \left|\int_0^T\int_D(\o{\alpha^{\e}}(u^\e(t))u^\e(t) - \o{\alpha^{\e}}(\o{u}(t))\o{u}(t)) \phi^\e\psi(t) dxdt \right|\leq\\
&C\E \int_0^T\|u^\e(t) - \o{u}(t) \|_{L^2(D)}dt, 
\end{split}
\end{equation}
based on the uniform Lipschitz condition of $\o{\alpha^{\e}}$ and the imbedding of $H_0^1(D)$ into $L^2(D)$. The uniform bounds for $u^\e$ now give \eqref{convre''}. 
\end{proof}
\subsection{Convergence of $S^\e_3$}
\label{subs44}
\begin{lemma}
\label{lemmaconv1'}
For fixed $\o{u}\in L^\infty (\Omega; C([0, T]; L^2(D)))$, $\phi^\e\in H_0^1(D)$ uniformly bounded and $\Psi\in C[0, T]$ let us define by $S^\e_3$ the integral $\displaystyle \int_0^T \int_D \left(\o{\alpha^{\e}}(\o{u}(t))\o{u}(t) - \o{\alpha}(\o{u}(t))\o{u}(t)\right)\phi^\e \psi (t)dxdt$. Then:
\begin{equation}
\label{convre'}
\begin{split}
\lim_{\e\to 0} \E\left| S^\e_3\right| =& 0. 
\end{split}
\end{equation}
\end{lemma}
\begin{proof}
For any $t\in[0, T]$ consider the sequence of functions $F^\e_t: L^2(D) \to L^2(D)$, 
$$F^\e_t(z) (x) = \left(\alpha \left(\dfrac{x}{\e}, z(x)\right) - \int_Y\alpha \left(y, z(x)\right) \right) u(t, x). $$
We show now that for any $z\in L^2(D)$, for every $t\in[0, T]$ and a. e. $\omega\in\Omega$, $F^\e_t (z)$ converges in $L^2(D)$ to $0$. We we fix $\omega$ and $t$ and let $z_n$ and $w_n$ two sequences of continuous functions converging in $L^2(D)$ to $z$ and $u(t)$. We use Lemma 1. 3 from \cite{A-2s} and obtain that the sequence $ F^\e_n(x) = \left(\alpha \left(\dfrac{x}{\e}, z_n(x)\right) - \displaystyle\int_Y\alpha \left(y, z_n(x)\right) \right) w_n(x)$ converges when $\e\to 0$ to $0$ in $L^2(D)$. 

But $$\left|F^\e_n(x) - F^\e_t(z)(x) \right| \leq c |w_n(x)- u(t, x)| + c|z_n(x)-z(x)|, $$
based on the Lipschitz condition and boundedness for $\alpha$. 
We deduce that $F^\e_t (z)$ converges in $L^2(D)$ to $0$. The sequence being also uniformly bounded by $\| C \o{u} \|_{L^\infty (\Omega; C([0, T]; L^2(D)))}$, Vitali's convergence theorem implies that the sequence of the integrals with respect to the probability measure on $L^2(D)$, $\mu^{\o{u}(t)}$ also converge to $0$ in $L^2(D)$:
$$\lim_{\e\to 0}\int_{L^2(D)} F^\e_t (z) d\mu^{\o{u}(t)} dz = 0 \ in \ L^2(D), $$
which can be rewritten as
$$\lim_{\e\to 0} \o{\alpha^\e}(\o{u}(t)) \o{u}(t) - \o{\alpha}(\o{u}(t)) \o{u}(t) = 0 \ in \ L^2(D). $$
This implies that $\mathbb{P}$ a. s. and for every $t\in[0, T]$
$$\lim_{\e\to 0} \int_D \left(\o{\alpha^\e}(\o{u}(t)) \o{u}(t)- \o{\alpha}(\o{u}(t)) \o{u}(t)\right)\phi \psi' (t)dx = 0, $$
with the sequence being also uniformly bounded. We apply the bounded convergence theorem and integrate over $\Omega \times [0, T]$ to get the result. 
\end{proof}
\subsection{Proof of Theorem \ref{thconv1}}
\label{subs45}
\begin{proof}
The uniform bounds \eqref{est2'} and \eqref{est3'} hold for $u^\e$. So the sequence is a. e. $\omega\in\Omega$ contained in a compact set $\mathcal{K}$ of $C([0, T];L^2(D))$  so the sequence is tight in $C([0, T];L^2(D)) $. Then, there exists a subsequence $u^{\e'}$ and a random element $\o{u} \in C([0, T];L^2(D))$ such that $u^{\e'}$ converges in distribution to $\o{u}$ in $C([0, T];L^2(D))$. Skorokhod theorem gives us the existence of a subsequence $u^{\e''}$ and another sequence $\wt{u^{\e''}}$ with the same distribution as $u^{\e''}$ defined on another probability space $\wt{\Omega}$ that converges point-wise to some $\wt{\o{u}}$, a random element of 
$C([0, T];L^2(D))$ with the same distribution as $\o{u}$. Since $u^{\e''}$ and $\wt{u^{\e''}}$ have the same distribution, then $\wt{u^{\e''}}$ is also bounded in $L^\infty(\wt{\Omega}, L^2(0, T;H_0^1(D))$. Hence, (up to another subsequence) and a.s. $\wt{u^{\e''}}$ converges to $\wt{\o{u}}$ weakly in $L^2(0, T;H_0^1(D))$. It   follows from here that a.s.,  $\wt{\o{u}}$ belongs to $\mathcal{K}$ so $\wt{\o{u}}\in L^\infty(\wt{\Omega}, L^2(0, T;H_0^1(D))$ and $\o{u}\in L^\infty(\Omega, L^2(0, T;H_0^1(D))$. 

In order to get the macroscopic equation for $\wt{\o{u}}$ we use the oscillating test function method of Tartar..., we use in the variational formulation \eqref{weaksole} for $u^{\e''}$ a test function $\phi^{\e''}$ of the form $\phi + \e'' \nabla \phi \cdot \chi^{*\e''}$ where $\phi \in C_0^\infty(D)$, multiply it with $\psi'$ where $\psi \in C_0^1(0, T)$ to get:
\begin{equation}
\label{weaksole'}
\begin{split}
&\int_0^T \int_D u^{\e''}(t) \phi^{\e''} \psi'(t) dx dt - \int_0^T\int_D u^{\e''}_0 \phi^{\e''} \psi'(t) dx dt -\int_0^T \int_D A^{\e''} \nabla u^{\e''}(t)\nabla\phi^{\e''}\psi(t)dx dt \\
+& \int_0^T \int_D \alpha^{\e''}(v^{\e''}(t)) u^{\e''}(t) \phi^{\e''}\psi(t)dx dt =-\int_0^T \int_D f(t) \phi^{\e''} \psi(t) dx dt. 
\end{split}
\end{equation}
We notice that
\begin{equation}
\label{convre'''}
\lim_{\e''\to 0} \E\left|\displaystyle\int_0^T \int_D \left(\alpha^{\e''}(v^{\e''}(t))u^{\e''}(t) - \o{\alpha}(\o{u}(t))\o{u}(t)\right)\phi^{\e''} \psi(t)dx dt \right| = 0. 
\end{equation}
We write:
\begin{equation}\nonumber
\begin{split}
\alpha^{\e''}(v^{\e''})u^{\e''} - \o{\alpha}(\o{u})\o{u} &= \alpha^{\e''}(v^{\e''})u^{\e''} - \o{\alpha^{\e''}}(u^{\e''})u^{\e''} + \o{\alpha^{\e''}}(u^{\e''})u^{\e''}-\o{\alpha^{\e''}}(\o{u})u^{\e''}\\
&+\o{\alpha^{\e''}}(\o{u})u^{\e''} -\o{\alpha^{\e''}}(\o{u})\o{u}+ \o{\alpha^{\e''}}(\o{u})\o{u} -\o{\alpha}(\o{u})\o{u}, 
\end{split}
\end{equation}
so
\begin{equation}\nonumber
\begin{split}
\int_0^T \int_D \alpha^{\e''}(v^{\e''}(t))u^{\e''}(t)\phi^{\e''} \psi(t) dx dt = S^{\e''}_1 + S^{\e''}_2 +S^{\e''}_3, 
\end{split}
\end{equation}
where
\begin{equation}
\label{s1}
S^{\e''}_1 = \int_0^T \int_D \left(\alpha^{\e''}(v^{\e''}(t)) - \o{\alpha^{\e''}}(u^{\e''}(t))\right)u^{\e''}(t)\phi^{\e''} \psi(t) dxdt, 
\end{equation}
\begin{equation}
\label{s2}
S^{\e''}_2 = \int_0^T \int_D \left( \o{\alpha^{\e''}}(u^{\e''}(t))u^{\e''}(t) - \o{\alpha^{\e''}}(\o{u}(t))\o{u}(t) \right)\phi^{\e''} \psi(t)dx dt, 
\end{equation}
and
\begin{equation}
\label{s3}
S^{\e''}_3 = \int_0^T \int_D \left(\o{\alpha^{\e''}}(\o{u}(t))\o{u}(t) - \o{\alpha}(\o{u}(t))\o{u}(t)\right)\phi^{\e''} \psi (t)dxdt. 
\end{equation}
Lemmas \ref{lemmaconv1}, \ref{lemmaconv1''} and \ref{lemmaconv1'} give that $\displaystyle\lim_{\e''\to 0 }\E|S^{\e''}_1|=\lim_{\e''\to 0 }\E|S^{\e''}_2|=\lim_{\e''\to 0 }\E|S^{\e''}_3|=0$ so we have \eqref{convre'''} which together with \eqref{weaksole'} gives 
\begin{equation}\label{eq6}
\begin{split}
\lim_{\e''\to 0} \E &\left|\int_0^T \int_D u^{\e''}(t) \phi^{\e''} \psi'(t) dx dt - \int_0^T\int_D u_0 \phi \psi'(t) dx dt - \int_0^T \int_D A^{\e''} \nabla u^{\e''}(t)\nabla\phi^{\e''}\psi(t)dx dt+\right. \\
&\left. \int_0^T \int_D\o{\alpha}(\o{u}(t))\o{u}(t)\phi^{\e''} \psi(t)dx dt+ \int_0^T \int_D f(t)\phi\psi(t) dxdt\right|=\\
\lim_{\e''\to 0} \wt\E &\left|\int_0^T \int_D \wt{u^{\e''}}(t) \phi^{\e''} \psi'(t) dx dt - \int_0^T\int_D u_0 \phi \psi'(t) dx dt - \int_0^T \int_D A^{\e''} \nabla \wt{u^{\e''}}(t)\nabla\phi^{\e''}\psi(t)dx dt+\right. \\
&\left. \int_0^T \int_D\o{\alpha}(\o{u}(t))\o{u}(t)\phi\psi(t)dx dt+ \int_0^T \int_D f(t)\phi\psi(t) dxdt\right|=0. 
\end{split}
\end{equation}
We make now several calculations under the integral in the above equation and then pass to the limit pointswise in $\wt{\omega} \in \wt{\Omega}$:
\begin{equation*}
\begin{split}
& \int_0^T \int_D A^{\e''} \nabla \wt{u^{\e''}}\nabla\left(\phi +\e''\nabla\phi\cdot \chi^{*\e''} \right)\psi(t)dxdt= \\
& \int_0^T \int_D A^{\e''} \nabla \wt{u^{\e''}}\left(\nabla\phi +\e''\nabla\nabla\phi \chi^{*\e''} +\e'' \nabla\phi \nabla\chi^{\e''}\right)\psi(t)dxdt= \\
& \int_0^T \int_D A^{\e''} \nabla \wt{u^{\e''}}\nabla\phi\psi(t) +\e''A^{\e''} \nabla \wt{u^{\e''}}\nabla\nabla\phi \chi^{*\e''}\psi(t) +\e'' A^{\e''} \nabla \wt{u^{\e''}}\nabla\phi \nabla\chi^{\e''*}\psi(t)dxdt= \\
& \int_0^T \int_D A^{\e''} \nabla \wt{u^{\e''}}\nabla\phi\psi(t) +\e''A^{\e''} \nabla \wt{u^{\e''}}\nabla\nabla\phi \chi^{*\e''}\psi(t) +\e'' A^{\e''}\nabla\chi^{\e''} \nabla \wt{u^{\e''}}\nabla\phi\psi(t)dxdt. 
\end{split}
\end{equation*}
From the equation \eqref{cellpre} satisfied by $\chi^{\e''}$ we have that
\begin{equation*}
\begin{split}
\int_D A^{\e''}\left( I+ \e'' \nabla \chi^{\e''} \right) \nabla \left(\wt{u^{\e''}} \nabla \phi \right) dx&=0\Rightarrow\\
\int_D \left(A^{\e''} \nabla \wt{u^{\e''}} \nabla \phi +\e'' A^{\e''} \nabla\chi^{\e''} \nabla \wt{u^{\e''}} \nabla \phi \right)dx&=-\int_D A^{\e''} \wt{u^{\e''}} \nabla\nabla \phi dx - \int_D \e'' A^{\e''}\nabla\chi^{\e''} \wt{u^{\e''}} \nabla\nabla\phi dx, 
\end{split}
\end{equation*}
so we get that
\begin{equation*}
\begin{split}
& \int_0^T \int_D A^{\e''} \nabla \wt{u^{\e''}}\nabla\left(\phi +\e''\nabla\phi\cdot \chi^{*\e''} \right)\psi(t)dxdt= \\
& \int_0^T \int_D \left(\e''A^{\e''} \nabla \wt{u^{\e''}}\nabla\nabla\phi \chi^{*\e''}\psi(t)- A^{\e''} \wt{u^{\e''}} \nabla\nabla \phi \psi(t)- \e'' A^{\e''}\nabla\chi^{\e''} \wt{u^{\e''}} \nabla\nabla\phi\psi(t)\right) dxdt=\\
& \int_0^T \int_D \left(\e''A^{\e''} \nabla \wt{u^{\e''}}\nabla\nabla\phi \chi^{*\e''}\psi(t)- A^{\e''}\left(I + \e''\nabla\chi^{\e''}\right) \wt{u^{\e''}} \nabla\nabla \phi \psi(t) \right)dxdt, 
\end{split}
\end{equation*}
and will converge pointwise in $\wt{\Omega}$ (see \cite{A-2s} Lemma 1. 3) to
$$ \int_0^T \int_D - \o{A} \wt{\o{u}} \nabla\nabla \phi \psi(t)dxdt = \int_0^T \int_D \o{A} \nabla \wt{\o{u}} \nabla \phi \psi(t)dxdt. $$
The sequence given in \eqref{eq6} above converges in $L^1(\wt{\Omega})$ to $0$ but also pointwise in $\wt{\Omega}$ to 
$$\int_0^T \int_D \left( \wt{\o{u}}(t)\phi\psi'(t) - u_0 \phi\psi' (t) -\o{A} \nabla \wt{\o{u}}\nabla\phi\psi(t) 
+ f(t)\phi\psi(t) + \o{\alpha} (\wt{\o{u}}(t)) \wt{\o{u}}(t) \phi\psi(t)\right)dx dt, $$
which means that $\wt{\o{u}}$ is pointwise the weak solution of the deterministic equation \eqref{eqou1} which, according to Theorem \ref{thexunou} has a unique solution, so $\wt{\o{u}}$ and $\o{u}$ are deterministic. Then, the whole sequence $u^{\e''}$ converges to $\o{u}$ in distribution, and since $\o{u}$ is deterministic then the convergence is also in probability see \cite{JP} Theorem 18.3. 
\end{proof}
\section*{Acknowledgements}
Hakima Bessaih was partially supported by NSF grant DMS-1418838.


\begin{thebibliography}{99}

\bibitem{A-2s}
G. Allaire, {\it Homogenization and two-scale convergence}, SIAM J. Math. Anal., {\bf 23}(6), 
(1992), pp. 1482--1518. 

 \bibitem{BEM1}
 H. Bessaih, Y. Efendiev, F. Maris, {\it Homogenization of Brinkman flows in heterogenous dynamic media}, SPDE: Analysis and Computations, {\bf 3} (2015), no 4, 479--505.

\bibitem{Cerrai}
S. Cerrai, Second order PDE's in finite and infinite dimension. 
A probabilistic approach., {\it Lecture Notes in Mathematics. }, 1762. Springer-Verlag, Berlin, (2001). 

\bibitem{Cerrai2009}
S. Cerrai, {\it A Khasminskii type averaging principle for stochastic reaction-diffusion equations}, 
Ann. Appl. Probab., {\bf 19} (2009), no. 3, 899--948. 

\bibitem{CF2008}
S. Cerrai, M. Freidlin, {\it Averaging principle for a class of stochastic reaction-diffusion equations}, 
 Probab. Theory Related Fields ., {\bf 144}  (2009), no. 1-2, 137--177.
 
\bibitem {DPZ2}
G. Da Prato, J. Zabczyk:
\textit{Ergodicity for infinite-dimensional systems}, 
London Mathematical Society Lecture Note Series, 229. 
Cambridge University Press, Cambridge (1996). 

\bibitem{FW1}
M. Freidlin, A. Wentzell, {\it Averaging principle for stochastic perturbations of multifrequency systems}, 
Stochastics and Dynamics, {\bf 3} (2003), 393--408.

\bibitem{JP}
J. Jacod, P. Protter, {\it Probability Essentials}, Universitext, Springer-Verlag, Berlin (2000). 

\bibitem {pazy} A. Pazy: \textit{Semigroups of Linear Operators and
Applications to Partial Differential Equations}, Applied
Mathematical Sciences, {\bf 44}, Springer-Verlag, New York (1983). 

\bibitem {temam} R. Temam, {\it Navier-Stokes equations. Theory and numerical analysis. 
Studies in Mathematics and its Applications 2}, North-Holland Publishing Co., Amsterdam-New York (1979). 

\bibitem {Yosida} K. Yosida, {\it Functional Analysis. Reprint of the sixth (1980) edition. Classics in Mathematics. }, Springer-Verlag, Berlin 11 (1995): 14. 

\end{thebibliography}
\end{document}